\font \sevenrm=cmr7
\font \fiverm=cmr5
\documentclass[12pt, english]{amsart}
\input epsf.sty
\usepackage{amsfonts}
\usepackage{amssymb}
\usepackage{amsmath}
\usepackage{amsthm}
\usepackage{amscd}
\usepackage[all]{xy} 
\usepackage{epsfig,exscale}
\usepackage{color}
\usepackage{graphicx}
\usepackage{xspace}
\usepackage{axodraw4j}
\usepackage{colordvi}

\newcommand{\nc}{\newcommand}

{\everymath{\displaystyle\everymath{}}\array}%
{\endarray}
\newtheorem{theorem}{Theorem}
\newtheorem{definition}{Definition}
\newtheorem{corollary}{Corollary}

\newtheorem{proposition}{Proposition}
\newtheorem{ex}{Example}
\newtheorem{remark}{Remark}
\nc{\comment}[1]{[[{\tt #1}]] }
\nc{\Cal}[1]{{\mathcal {#1}}}
\nc{\mop}[1]{\mathop{\hbox {\rm #1} }\nolimits}
\nc{\gmop}[1]{\mathop{\hbox {\bf #1} }\nolimits}

\def\starg{{\displaystyle\mathop{\star}\limits}_g}
\def\trianglerightg{{\displaystyle\mathop{\triangleright}\limits}_g}
\def\trianglerightv{{\displaystyle\mathop{\triangleright}\limits}_v}
\def\trianglerights{{\displaystyle\mathop{\triangleright}\limits}_s}
\def\trianglerightl{{\displaystyle\mathop{\triangleright}\limits}_l}
\def\trianglerightr{{\displaystyle\mathop{\triangleright}\limits}_r}
\def\cupv{{\displaystyle\mathop{\cup}\limits}_v}
\def\cups{{\displaystyle\mathop{\cup}\limits}_s}

\nc{\smop}[1]{\mathop{\hbox {\sevenrm #1} }\nolimits}
\nc{\ssmop}[1]{\mathop{\hbox {\fiverm #1} }\nolimits}
\nc{\mopl}[1]{\mathop{\hbox {\rm #1} }\limits}
\nc{\smopl}[1]{\mathop{\hbox {\sevenrm #1} }\limits}
\nc{\ssmopl}[1]{\mathop{\hbox {\fiverm #1} }\limits}
\nc{\frakg}{{\frak g}}
\nc{\g}[1]{{\frak {#1}}}
\def \restr#1{\mathstrut_{\textstyle |}\raise-6pt\hbox{$\scriptstyle #1$}}
\def \srestr#1{\mathstrut_{\scriptstyle |}\hbox to
-1.5pt{}\raise-4pt\hbox{$\scriptscriptstyle #1$}}
\nc{\wt}{\widetilde} \nc{\wh}{\widehat}
\nc{\redtext}[1]{\textcolor{red}{#1}}
\nc{\bluetext}[1]{\textcolor{blue}{#1}}
\nc\fleche[1]{\mathop{\hbox to #1 mm{\rightarrowfill}}\limits}
\nc{\ignore}[1]{}
\def\semi{\mathrel{\times}\kern -.85pt\joinrel\mathrel{\raise
1.4pt\hbox{${\scriptscriptstyle |}$}}}
\nc\R{{\mathbb R}}
\nc\N{{\mathbb N}}
\nc\inver{^{-1}}
\nc\point{\hbox{\bf .}}
\nc\un{\hbox{\bf 1}}

\def\graphearetenor{\,{\scalebox{0.25}{
\begin{picture}(130,6) (175,-221)
\SetWidth{3.0}
\SetColor{Black}
\Line(176,-218)(304,-218)
\end{picture}
}}\,}

\def\graphearetecurcroixun{\,{\scalebox{0.25}{
\begin{picture}(130,90) (111,-177)
\SetWidth{3.0}
\SetColor{Black}
\Photon(112,-164)(240,-164){7.5}{6}
\Line(191.998,-179.998)(160.002,-148.002)\Line(160.002,-179.998)(191.998,-148.002)
\Text(176,-212)[lb]{\Huge{\Black{$1$}}}
\end{picture}
}}\,}

\def\graphearetenorcroixzero{\,{\scalebox{0.25}{
\begin{picture}(130,34) (95,-143)
\SetWidth{3.0}
\SetColor{Black}
\Line(96,-126)(224,-126)
\Line(144.002,-110.002)(175.998,-141.998)\Line(175.998,-110.002)(144.002,-141.998)
\Text(156,-172)[lb]{\Huge{\Black{$0$}}}
\end{picture}
}}\,}
\def\graphearetenorcroixun{\,{\scalebox{0.25}{
\begin{picture}(130,34) (95,-143)
\SetWidth{3.0}
\SetColor{Black}
\Line(96,-126)(224,-126)
\Line(144.002,-110.002)(175.998,-141.998)\Line(175.998,-110.002)(144.002,-141.998)
\Text(156,-172)[lb]{\Huge{\Black{$1$}}}
\end{picture}
}}\,}
\def\graphearetecur{\,{\scalebox{0.25}{
\begin{picture}(133,20) (104,-209)
\SetWidth{3.0}
\SetColor{Black}
\Photon(105,-201)(236,-197){7.5}{7}
\end{picture}
}}\,}
\def\graphev{\,{\scalebox{0.25}{
\begin{picture}(66,34) (191,-191)
    \SetWidth{3.0}
    \SetColor{Black}
    \Line(192,-174)(224,-174)
    \Line(224,-174)(256,-158)
    \Line(224,-174)(256,-190)
\end{picture}}}\,}
\def\graphevcur{\,{\scalebox{0.15}{
\begin{picture}(160,166) (97,-122)
\SetWidth{3.0}
\SetColor{Black}
\Photon(98,-99)(195,-97){7.5}{5}
\Line(195,-99)(256,-37)
\Line(196,-100)(255,-161)
\end{picture}
}}\,}
\def\grapheexmdel{\,{\scalebox{0.25}{
\begin{picture}(162,52) (175,-209)
\SetWidth{3.0}
\SetColor{Black}
\Arc(256,-222)(16,270,630)
\Arc(256,-192)(34,-61.928,241.928)
\Line(288,-190)(336,-190)
\Line(224,-190)(176,-190)
\end{picture}
}}\,}
\def\graphecerc{\,{\scalebox{0.25}{
\begin{picture}(130,34) (191,-287)
\SetWidth{3.0}
\SetColor{Black}
\Arc(256,-270)(16,270,630)
\Line(272,-270)(320,-270)
\Line(240,-270)(192,-270)
\end{picture}
}}\,}
\def\graphecerccroizero{\,{\scalebox{0.25}{
\begin{picture}(162,63) (175,-101)
\SetWidth{3.0}
\SetColor{Black}
\Arc(256,-75)(35.777,243,603)
\Line(288,-75)(336,-75)
\Line(224,-75)(176,-75)
\Line(240.002,-122.998)(271.998,-91.002)\Line(240.002,-91.002)(271.998,-122.998)
\Text(256,-139)[lb]{\Huge{\Black{$0$}}}
\end{picture}
}}\,}
\def\graphecerccroiun{\,{\scalebox{0.25}{
\begin{picture}(162,63) (175,-101)
\SetWidth{3.0}
\SetColor{Black}
\Arc(256,-75)(35.777,243,603)
\Line(288,-75)(336,-75)
\Line(224,-75)(176,-75)
\Line(240.002,-122.998)(271.998,-91.002)\Line(240.002,-91.002)(271.998,-122.998)
\Text(256,-139)[lb]{\Huge{\Black{$1$}}}
\end{picture}
}}\,}
\def\graphedeltasansindice{\,{\scalebox{0.20}{
\begin{picture}(226,68) (111,-161)
    \SetWidth{3.0}
    \SetColor{Black}
    \Arc(224,-142)(48,180,540)
    \Line(192,-110)(192,-174)
    \Line(240,-94)(240,-190)
    \Line(272,-142)(336,-142)
    \Line(176,-142)(112,-142)
\end{picture}
}}\,}

\def\graphedeltacontratddd{\,{\scalebox{0.20}{
\begin{picture}(96,57) (336,-297)
    \SetWidth{3.0}
    \SetColor{Black}
\Arc(382.723,-358.862)(114.921,68.415,112.902)
    \Arc(379,-191.833)(131.078,-108.688,-67.576)
    \Line(354,-248)(354,-320)
    \Line(412,-249)(415,-317)
\end{picture}
  }}\,}
\def\graphedeltacontratdddd{\,{\scalebox{0.20}{
\begin{picture}(168,29) (184,-180)
    \SetWidth{3.0}
    \SetColor{Black}
    \Arc(243,-176)(22.627,135,495)
    \Arc(288,-175)(22.627,135,495)
    \Line(312,-176)(351,-176)
    \Line(221,-177)(185,-177)
\end{picture}
  }}\,}

 \def\qedj{\,{\scalebox{0.20}{ \begin{picture}(181,66) (218,-178)
    \SetWidth{3.0}
    \SetColor{Black}
    \PhotonArc(307.786,-179.786)(22.955,14.598,173.03){7.5}{3.5}
    \PhotonArc[clock](306,-170.142)(49.152,-175.498,-364.502){7.5}{8.5}
    \Line(219,-175)(398,-175)
  \end{picture}}}\,}
\def\qeda{\,{\scalebox{0.20}{
\begin{picture}(98,35) (271,-286)
    \SetWidth{3.0}
    \SetColor{Black}
    \PhotonArc(320,-277)(17.889,-26.565,206.565){7.5}{4.5}
    \Line(272,-285)(368,-285)
\end{picture}}}\,}

\def\qedcoiz{\,{\scalebox{0.20}
  { \begin{picture}(258,93) (191,-180)
    \SetWidth{3.0}
    \SetColor{Black}
    \PhotonArc(328,-187)(91.214,15.255,164.745){7.5}{12.5}
    \Line(192,-163)(448,-163)
    \Line(335.998,-178.998)(304.002,-147.002)\Line(304.002,-178.998)(335.998,-147.002)
    \Text(315,-210)[lb]{\Huge{\Black{$0$}}}
  \end{picture}}}\,}
\def\qedcoiu{\,{\scalebox{0.20}
  { \begin{picture}(258,93) (191,-180)
    \SetWidth{3.0}
    \SetColor{Black}
    \PhotonArc(328,-187)(91.214,15.255,164.745){7.5}{12.5}
    \Line(192,-163)(448,-163)
    \Line(335.998,-178.998)(304.002,-147.002)\Line(304.002,-178.998)(335.998,-147.002)
    \Text(315,-210)[lb]{\Huge{\Black{$1$}}}
  \end{picture}}}\,}
	
	\def\qedcoiz{\,{\scalebox{0.20}
  { \begin{picture}(258,93) (191,-180)
    \SetWidth{3.0}
    \SetColor{Black}
    \PhotonArc(328,-187)(91.214,15.255,164.745){7.5}{12.5}
    \Line(192,-163)(448,-163)
    \Line(335.998,-178.998)(304.002,-147.002)\Line(304.002,-178.998)(335.998,-147.002)
    \Text(315,-210)[lb]{\Huge{\Black{$0$}}}
  \end{picture}}}\,}

\def\pre3{\,{\scalebox{0.20}{  
\begin{picture}(226,162) (319,-80)
    \SetWidth{3.0}
    \SetColor{Black}
    \Arc(432,-13.962)(57.658,123.711,416.289)
    \Line(320,-14)(374,-15)
    \Line(490,-14)(544,-14)
    \Arc(432,34)(32,180,540)
    \Line(425.001,62.001)(434.999,69.999)\Line(426.001,70.999)(433.999,61.001)
    \Line(430.001,-5.999)(437.999,7.999)\Line(427.001,4.999)(440.999,-2.999)
    \Line(435.001,-78.999)(444.999,-65.001)\Line(433.001,-67.001)(446.999,-76.999)
    \Text(440,75)[lb]{\Huge{\Black{$0$}}}
    \Text(445,-15)[lb]{\Huge{\Black{$1$}}}
    \Text(445,-91)[lb]{\Huge{\Black{$0$}}}
  \end{picture}}}\,}  
		
\def\oudam{\,{\scalebox{0.20}{ 
  \begin{picture}(210,92) (191,-149)
    \SetWidth{3.0}
    \SetColor{Black}
    \Arc(288,-132)(16,180,540)
    \PhotonArc[clock](288,-130)(64.031,-178.21,-361.79){7.5}{10.5}
    \Line(192,-132)(272,-132)
    \Line(307,-140)(400,-132)
   \end{picture}}}\,} 
		
  \def\prelied{\,{\scalebox{0.20}{ 
	\begin{picture}(258,110) (175,-163)
    \SetWidth{3.0}
    \SetColor{Black}
    \PhotonArc[clock](296,-133.993)(72.994,-170.532,-369.468){7.5}{13.5}
    \Line(176,-146)(432,-146)
    \Line(240.002,-161.998)(271.998,-130.002)\Line(240.002,-130.002)(271.998,-161.998)
    \Line(335.998,-130.002)(304.002,-161.998)\Line(335.998,-161.998)(304.002,-130.002)
    \Text(255,-170)[lb]{\Huge{\Black{$1$}}}
    \Text(325,-175)[lb]{\Huge{\Black{$0$}}}
 \end{picture}}}\,} 

\def\preliekb{\,{\scalebox{0.20}{ 
\begin{picture}(328,96) (243,-83)
    \SetWidth{3.0}
    \SetColor{Black}
    \Arc(346,-74.839)(24.549,-12.136,192.136)
    \PhotonArc(433.74,-86.565)(35.554,7.377,167.715){7.5}{5.5}
    \PhotonArc[clock](406.876,-127.443)(131.948,159.855,21.073){7.5}{16.5}
    \Line(244,-82)(570,-82)
 \end{picture}}}\,} 
	
	\def\prelz{\,{\scalebox{0.20}{
	\begin{picture}(221,65) (246,-170)
    \SetWidth{3.0}
    \SetColor{Black}
    \PhotonArc[clock](357.533,-197.483)(83.883,154.219,24.273){7.5}{10.5}
    \Arc[clock](381.696,-150.652)(18.8,-146.604,-383.007)
    \Line(247,-160)(454,-160)
    \Line(450,-160)(466,-160)
    \Line(311.001,-168.999)(324.999,-153.001)\Line(310.001,-154.001)(325.999,-167.999)
    \Text(320,-180)[lb]{\Huge{\Black{$0$}}}
  \end{picture}}}\,} 
	
	\def\prelieu{\,{\scalebox{0.20}{
	\begin{picture}(261,89) (240,-107)
    \SetWidth{3.0}
    \SetColor{Black}
    \PhotonArc[clock](336.9,-107.674)(31.613,167.812,10.34){7.5}{4.5}
    \PhotonArc(369.948,-111.282)(84.459,5.628,173.69){7.5}{12.5}
    \Line(241,-102)(500,-102)
    \Line(412.001,-105.999)(421.999,-96.001)\Line(412.001,-96.001)(421.999,-105.999)
    \Text(425,-130)[lb]{\Huge{\Black{$1$}}}
  \end{picture}}}\,}

	\def\grpddz{\,{\scalebox{0.20}{
	\begin{picture}(439,121) (77,-67)
    \SetWidth{3.0}
    \SetColor{Black}
    \Arc(194.875,-59.763)(62.883,13.578,165.482)
    \Arc(191,-68)(18.385,158,518)
    \Arc(191.5,-41.389)(28.448,-49.435,229.435)
    \PhotonArc(402.092,-67.43)(56.716,21.114,156.704){7.5}{7.5}
    \Line(78,-45)(162,-44)
    \Line(220,-44)(515,-45)
    \PhotonArc[clock](293.5,-138.605)(207.779,153.224,26.776){7.5}{23.5}
    \Line(291.001,-46)(300.999,-40)\Line(293,-38.001)(299,-47.999)
    \Line(395,-48)(399,-42)\Line(394,-43)(400,-47)
    \Text(300,-60)[lb]{\Huge{\Black{$0$}}}
    \Text(400,-60)[lb]{\Huge{\Black{$0$}}}
  \end{picture}}}\,}

	\def\grpddze{\,{\scalebox{0.20}{
	\begin{picture}(375,121) (140,-67)
    \SetWidth{3.0}
    \SetColor{Black}
    \Arc(321,-88)(18.028,177,537)
    \Arc[clock](320.662,-65.051)(27.784,-124.312,-405.89)
    \Arc[clock](226.875,-79)(35.05,165.121,13.194)
    \PhotonArc(409.68,-83.972)(39.85,20.524,161.003){7.5}{5.5}
    \Line(141,-70)(294,-69)
    \Line(347,-70)(514,-71)
    \PhotonArc[clock](323,-153.441)(176.444,152.145,27.855){7.5}{19.5}
    \Line(221.001,-71)(228.999,-67)\Line(223,-65.001)(227,-72.999)
    \Line(405,-72)(411,-66)\Line(405,-66)(411,-72)
    \Text(235,-85)[lb]{\Huge{\Black{$0$}}}
    \Text(420,-85)[lb]{\Huge{\Black{$0$}}}
  \end{picture}}}\,}

	\def\grdtz{\,{\scalebox{0.20}{
	\begin{picture}(341,123) (226,-61)
    \SetWidth{3.0}
    \SetColor{Black}
    \Arc[clock](305,-55.923)(31.247,172.788,7.212)
    \PhotonArc(453.653,-66.21)(43.406,17.718,163.662){7.5}{6.5}
    \Line(227,-54)(566,-54)
    \PhotonArc[clock](392.874,-89.525)(143.346,165.651,13.937){7.5}{19.5}
    \Line(301.001,-56.999)(310.999,-49.001)\Line(302.001,-48.001)(309.999,-57.999)
    \Line(365.001,-56.999)(374.999,-49.001)\Line(366.001,-48.001)(373.999,-57.999)
    \Line(449.001,-56)(456.999,-50)\Line(450,-49.001)(456,-56.999)
    \Text(315,-70)[lb]{\Huge{\Black{$0$}}}
    \Text(380,-70)[lb]{\Huge{\Black{$0$}}}
    \Text(465,-70)[lb]{\Huge{\Black{$0$}}}
  \end{picture}}}\,}

\def\diagramme #1{\vskip 4mm \centerline {#1} \vskip 4mm}
\begin{document}
\title{
{On the pre-Lie algebra of specified Feynman graphs }}

\author{Mohamed Belhaj Mohamed}
\address{{Mathematics Departement, Sciences college, Taibah University, Kingdom of Saudi Arabia~.}\vspace{0.01cm}
{Laboratoire de math\'ematiques physique fonctions sp\'eciales et applications, Universit\'e de Sousse, rue Lamine Abassi 4011 H. Sousse,  Tunisie.}}     
        \email{mohamed.belhajmohamed@isimg.tn}

\date{December 2018}
\noindent{\footnotesize{${}\phantom{a}$ }}
\begin{abstract}
We study the pre-Lie algebra of specified Feynman graphs $\wt{V}_{\Cal T}$ and we define a pre-Lie structure on its doubling space $\wt{\Cal F}_{\Cal T}$. We prove that $\wt{\Cal F}_{\Cal T}$ is pre-Lie module on $\wt{V}_{\Cal T}$ and we find some relations between the two pre-Lie structures. Also, we study the enveloping algebras of two pre-Lie algebras denoted respectively by $(\wt {\Cal D'}_{\Cal T}, \bigstar, \Phi)$ and $(\wt {\Cal H'}_{\Cal T}, \star, \Psi)$ and we prove that $(\wt {\Cal D'}_{\Cal T}, \bigstar, \Phi)$ is a module-bialgebra on $(\wt {\Cal H'}_{\Cal T}, \star, \Psi)$. 
\end{abstract}
\maketitle

\tableofcontents
\section{Introduction}  
Hopf algebras of Feynman graphs have been studied by A. Connes and D. Kreimer in \cite{A.D2000}, \cite{ad01}, \cite{ad98}and \cite{dk98} as a powerful tool to explain the combinatorics of renormalization in quantum field theory, and it appeared thereafter that in the works of K. Ebrahimi-Fard, D. Manchon \cite{Dm11, Dm08}, van Suijlekom \cite{wvs, wvs06} and many others.

In \cite{mb}, we have introduced the concept of doubling bialgebra of specified Feynman graphs to give a sense to some divergent integrals given by a Feynman graphs. It is given by the vector space $\wt{\Cal D}_{\Cal T}$ spanned by the pairs $(\bar\Gamma, \bar\gamma)$ of locally $1PI$ specified graphs, with $\bar\gamma \subset \bar\Gamma$ and $\bar\Gamma / \bar\gamma \in \wt {\Cal H}_{\Cal T}$. The product $m$ is again given by juxtaposition:
$$
m \big((\bar\Gamma_1, \bar\gamma_1)\otimes (\bar\Gamma_2, \bar\gamma_2)\big) = (\bar\Gamma_1  \bar\Gamma_2, \bar\gamma_1 \bar\gamma_2),
$$
and the coproduct $\Delta$ is defined as follows: 
$$
\Delta (\bar\Gamma, \bar\gamma ) = \sum_{\substack{\bar\delta \subseteq \bar\gamma \\ \bar\gamma / \bar\delta \in \Cal T }} ( \bar \Gamma, \bar\delta )  \otimes ( \bar\Gamma / \bar\delta, \bar\gamma / \bar\delta).
$$
We have also studied, in collaboration with Dominique Manchon \cite{DB}, the notion of doubling bialgebra in the context of rooted trees, we have defined the doubling bialgebras of rooted trees given by extraction contraction and admissible cuts, and we have shown the existence of many relations between these two structures.

Pre-Lie algebra of insertion  was studied by A. Connes and D. Kreimer \cite{A.D2000} in the context of Feynman graphs and F. Chapoton and M. Livernet in context of rooted trees \cite{cl} as being the space of primitive elements in the graded dual of a right-sided Hopf algebra. The Hopf algebra of specified Feynman graphs  is also right-sided so we can construct a pre-Lie  structure on its graded dual.

In this article, we star by describing this pre-Lie structure. On the space of connected specified Feynman graphs $\wt {V}_{\Cal T}$ we define the pre-Lie product  $\triangleright$ by insertion of graphs. For all $\bar \Gamma_1 = (\Gamma_1 , \underline{i}), \bar \Gamma_2 = (\Gamma_2 , \underline{j}) \in \wt {V}_{\Cal T}$ we have: 
$$
(\Gamma_1 , \underline{i}) \triangleright (\Gamma_2 , \underline{j})  := \sum_{v \in \Cal V (\Gamma_2)} (\Gamma_1 , \underline{i}) \trianglerightv (\Gamma_2 , \underline{j}),
$$
where $(\Gamma_1 , \underline{i}) \trianglerightv (\Gamma_2 , \underline{j})$ is the graph obtained by replacing the vertex $v$ by the graph $\Gamma_1$ in $\Gamma_2$.

In the third section,  we find a pre-Lie structure on the doubling space of connected specified graphs $\wt{V}_{\Cal T}$ noted $\wt{\Cal F}_{\Cal T}$. The pre-Lie product is defined for all $(\bar \Gamma_1 , \bar \gamma_1)$, $(\bar \Gamma_2, \bar \gamma_2)$ in  $\wt{\Cal F}_{\Cal T}$ by:
$$
(\bar \Gamma_1 , \bar \gamma_1)  \odot (\bar \Gamma_2, \bar \gamma_2) := \sum_{\substack{v \in \Cal V(\Gamma_2 - \gamma_2)\; , v = \smop{res} \Gamma_1}}  (\bar\Gamma_1 \trianglerightv \bar\Gamma_2 , \bar\gamma_1 \bar\gamma_2),
$$
where $v \in \Cal V(\Gamma_2 - \gamma_2)$ denotes that $v$ is a vertex of $\Gamma_2$ but its not a vertex of $\gamma_2$ and $\mop{res}\Gamma_1$ denotes the residue of the graph $\Gamma_1$. We also prove that $\wt{\Cal F}_{\Cal T}$ is a pre-Lie module on $\wt{V}_{\Cal T}$ where the action is given for all $\bar \Gamma_1 \in \wt{V}_{\Cal T}$ and $(\bar \Gamma_2, \bar \gamma_2) \in \wt{\Cal F}_{\Cal T}$ by:
$$
\bar \Gamma_1 \rightarrow (\bar \Gamma_2, \bar \gamma_2) := \sum_{\substack{v \in \Cal V(\gamma_2)\; , v = \smop{res} \gamma_1}}  (\bar\Gamma_1 \trianglerightv \bar\Gamma_2 , \bar\Gamma_1 \trianglerightv \bar\gamma_2). 
$$
We also give some relations between the action $\rightarrow$ and the pre-Lie product $\odot$. We show that the action $\rightarrow$ is a derivation of the algebra $(\wt{\Cal F}_{\Cal T}, \odot)$. In other words, for any $\bar \Gamma_1 \in \wt {V}_{\Cal T}$ and  $(\bar \Gamma_2, \bar \gamma_2), (\bar \Gamma_3, \bar \gamma_3) \in \wt {\Cal F}_{\Cal T}$, we have:
$$
\bar \Gamma_1 \rightarrow  \big((\bar \Gamma_2, \bar \gamma_2) \odot (\bar \Gamma_3, \bar \gamma_3) \big) = \big(\bar \Gamma_1 \rightarrow (\bar \Gamma_2, \bar \gamma_2) \big) \odot (\bar \Gamma_3, \bar \gamma_3) + (\bar \Gamma_2, \bar \gamma_2) \odot \big(\bar \Gamma_1\rightarrow (\bar \Gamma_3, \bar \gamma_3) \big),
$$
and we show that the following diagram is commutative:
\diagramme{
\xymatrix{
\wt{V}_{\Cal T} \otimes \wt{\Cal F}_{\Cal T} \ar[d]_{I\otimes P_2} \ar[rrr]^{\rightarrow} 
&&&\wt{\Cal F}_{\Cal T} \ar[d]^{P_2}\\
\wt{V}_{\Cal T}\otimes\wt{\Cal H}_{\Cal T} \ar[rrr]_{\triangleright} 
&&& \wt{\Cal H}_{\Cal T} }
}
which means that the projection on the second component $P_2$ is a morphism of pre-Lie modules.

In the last section, we use the method of Oudom and Guin  \cite{og} to find the enveloping algebra of the pre-Lie algebra. We start by finding the enveloping algebra of the pre-Lie algebra of specified Feynman graphs. Starting from the pre-Lie algebra $(\wt {V}_{\Cal T}, \triangleright)$, we consider the Hopf symmetric algebra $\wt {\Cal H'}_{\Cal T} : = \Cal S (\wt {V}_{\Cal T})$ equipped with its usual unshuffling coproduct $\Psi$ and a product $\star$ coming from the pre-Lie structure:
$$
\bar\Gamma \star \bar\Gamma' = \sum_{(\Gamma)}  \bar\Gamma^{(1)} (\bar\Gamma^{(2)} \triangleright \bar\Gamma').
$$

By construction, the space $(\wt {\Cal H'}_{\Cal T}, \star, \Psi)$ is a Hopf algebra which is isomorphic to the enveloping algebra $\Cal U(\wt V_{Lie})$ of the pre-Lie algebra $\wt {V}_{Lie}$. We prove that $(\wt {\Cal H'}_{\Cal T}, m, \Psi)$ is a comodule-coalgebra on $(\wt {\Cal H}_{\Cal T},  m, \Delta)$, where $\wt {\Cal H}_{\Cal T}$ and $\wt {\Cal H'}_{\Cal T}$ are isomorphic as algebras.

Similarly, we construct the enveloping algebra of doubling pre-Lie algebra of specified Feynman graphs. Starting from the pre-Lie algebra $(\wt {\Cal F}_{\Cal T}, \odot)$, we consider the Hopf symmetric algebra $\wt {\Cal D'}_{\Cal T} : = \Cal S (\wt {\Cal F}_{\Cal T})$ equipped with its usual unshuffling coproduct $\phi$ and the product $\star$ coming from the pre-Lie structure:
$$
(\bar\Gamma, \bar\gamma) \bigstar (\bar\Gamma', \bar\gamma') = \sum_{(\Gamma, \gamma)}  (\bar\Gamma, \bar\gamma)^{(1)} \big((\bar\Gamma, \bar\gamma)^{(2)} \odot (\Gamma', \gamma')\big).
$$

By construction, the space $(\wt {\Cal D'}_{\Cal T}, \bigstar, \phi)$ is a Hopf algebra which is isomorphic to the enveloping Hopf algebra $\Cal U(\wt {\Cal F}_{Lie})$ of the Lie algebra $\wt {\Cal F}_{Lie}$, and we finish by proving that $(\wt {\Cal D'}_{\Cal T}, m, \Phi)$ is a comodule-coalgebra on $(\wt {\Cal D}_{\Cal T},  m, \Delta)$, where $\wt {\Cal D}_{\Cal T}$ and $\wt {\Cal D'}_{\Cal T}$ are isomorphic as algebras.

 At the end, we give a relation between the two hopf structures $(\wt {\Cal D'}_{\Cal T}, \bigstar, \Phi)$ and $(\wt {\Cal H'}_{\Cal T}, \star, \Psi)$, we prove that $(\wt {\Cal D'}_{\Cal T}, \bigstar, \Phi)$ is a module-bialgebra on $(\wt {\Cal H'}_{\Cal T}, \star, \Psi)$.

\vspace{1cm}
\noindent
{\bf Acknowledgements:} I would like to thank Dominique Manchon for support and advice.
\section{Feynman graphs}

\subsection{Basic definitions}
A Feynman graph is a graph with a finite number of vertices and edges, which can be internal or external. An internal edge is an edge connected at both ends to a vertex, an external edge is an edge with one open end, the other end being connected to a vertex. The edges are obtained by using half-edges.
More precisely, let us consider two finite sets $\Cal V$ and $\Cal E$. A graph $\Gamma$ with $\Cal V$ (resp. $\Cal E$) as set of vertices (resp. half-edges) is defined as follows: let $\sigma : \Cal E \longrightarrow \Cal E $ be an involution and $\partial : \Cal E  \longrightarrow \Cal V$. For any vertex $v\in \Cal V$ we denote by $st(v) = \{e \in \Cal E / \partial (e) = v \}$ the set of half-edges adjacent to $v$. The fixed points of $\sigma$ are the \textsl {external edges} and the \textsl {internal edges} are given by the pairs $\{e , \sigma (e)\}$ for $e \neq \sigma (e)$. The graph $\Gamma$ associated to these data is obtained by attaching half-edges $e\in st(v)$ to any vertex $v\in\Cal V$, and joining the two half-edges $e$ and $\sigma(e)$ if $\sigma(e)\not =e$.\\

Several types of half-edges will be considered later on: the set $\Cal E$ is partitioned into several pieces $\Cal E_i$. In that case we ask that the involution $\sigma$ respects the different types of half-edges, i.e. $\sigma(\Cal E_i)\subset \Cal E_i$.\\

We denote by $\Cal I(\Gamma)$ the set of internal edges and by $\mop{Ext}(\Gamma)$ the set of external edges. The \textsl {loop number} of a graph $\Gamma$ is given by: 
$$L(\Gamma) =  \left|\Cal I (\Gamma)\right| - \left|\Cal V (\Gamma)\right|  + \left|\pi_0 (\Gamma)\right|,$$
where $\pi_0 (\Gamma)$ is the set of connected components of $\Gamma$.\\

A one-particle irreducible graph (in short, $1PI$ graph) is a connected graph which remains connected when we cut any internal edge. A disconnected graph is said to be locally $1PI$ if any of its connected components is $1PI$.

A covering subgraph of $\Gamma$ is a Feynman graph $\gamma$ (not necessarily connected), obtained from $\Gamma$ by cutting internal edges. In other words:
\begin{enumerate}
\item $ \Cal V (\gamma) = \Cal V (\Gamma)$. 
\item $ \Cal E (\gamma) =  \Cal E (\Gamma)$.
\item $ \sigma_\Gamma (e) = e \Longrightarrow \sigma_\gamma (e) = e$.
\item If $ \sigma_\gamma (e) \neq  \sigma_\Gamma (e) \;\; \text{then} \;\; \sigma_\gamma (e) = e \;\; \text{and} \;\; \sigma_\gamma (\sigma_\Gamma (e)) = \sigma_\Gamma (e)$.
\end{enumerate}

For any covering subgraph $\gamma$, the contracted graph $\Gamma / \gamma$ is defined by shrinking all connected components of $\gamma$ inside $\Gamma$ onto a point.  

The residue of the graph $\Gamma$, denoted by $\mop{res} \Gamma$, is the contracted graph $\Gamma / \Gamma$.

The skeleton of a graph $\Gamma$ denoted by $\mop {sk} \Gamma$ is the graph obtained by cutting all internal edges.
\subsection {Quantum field theory and specified graphs } 
We will work inside a physical theory $\Cal T$ ($\varphi^3 , \; \varphi^4$, QED, QCD  etc). The particular form of the Lagrangian leads to consider certain types of vertices and edges. A difficulty appears: the type of half-edges of $st (v)$ is not sufficient to determine the type of the vertex $v$. We denote by $\Cal E (\Cal T)$ the set of possible types of half-edges and by $\Cal V(\Cal T)$ the set of possible types of vertices.
\begin{ex}
$\Cal E ( \varphi^3) = \{ \graphearetenor \} \;\;\;, \;\;\; \Cal E ( QED) = \{ \graphearetenor  , \graphearetecur \}$.\\
${\Cal V}(\varphi^3) = \{ \graphearetenorcroixzero , \graphearetenorcroixun , \graphev  \} \;\;\;, \;\;\;{\Cal V}(QED) = \{ \graphearetenorcroixzero , \graphearetenorcroixun , \graphevcur , \graphearetecurcroixun \}$. 
\end{ex}
\begin{definition} \label{df1}
A specified graph of theory $\Cal T$ is a couple $(\Gamma , \underline{i})$ where:
\begin{enumerate}
\item $\Gamma$ is a locally $1PI$ superficially divergent graph (the residue of $\Gamma$ is an element of $\Cal T$) with half-edges and vertices of the type prescribed in $\Cal T$. 
\item $\underline{i} : \pi_0(\Gamma) \longrightarrow \mathbb N$, the values of $\underline {i}(\gamma)$ being prescribed by the possible types of vertex obtained by contracting the connected component $\gamma$ on a point.
\end{enumerate}
We will say that $(\gamma , \underline{j})$ is a specified covering subgraph of $(\Gamma , \underline{i})$, and we note $\big( (\gamma , \underline{j}) \subset (\Gamma , \underline{i}) \big)$  if:
\begin{enumerate}
\item $\gamma$ is a covering subgraph of $\Gamma$.
\item if $\gamma_0$ is a full connected component of $\gamma$, i.e if $\gamma_0$ is also a connected component of $\Gamma$, then $\underline{j} (\gamma_0) = \underline{i} (\gamma_0)$.
\end{enumerate}
\end{definition}
\begin{remark}
Sometimes we denote by $\bar \Gamma = (\Gamma , \underline{i})$ the specified graph, and we will write $\bar \gamma \subset \bar \Gamma$ for $(\gamma , \underline{j}) \subset (\Gamma , \underline{i})$.
\end{remark}
\begin{definition}
Let be $(\gamma , \underline{j}) \subset (\Gamma , \underline{i})$. The contracted specified subgraph is written:
$$\Gamma / \bar \gamma = (\Gamma/\bar\gamma , \underline{i}),$$
where $\bar\Gamma/\bar\gamma$ is obtained by contracting each connected component of $\gamma$ on a point, and specifying the vertices obtained with $\underline {j}$.
\end{definition}
\begin{remark}
The specification $\underline{i}$ is the same for the graph $\bar\Gamma$ and the contracted graph $\bar\Gamma / \bar \gamma$. 
\end{remark}

\subsection {Hopf algebras of Feynman graphs}
%
Let $ \wt V_{\Cal T}$ be the vector space generated by $1PI$ connected specified graphs $\Gamma$ with edges in $\Cal E (\Cal T)$ and vertices in $\Cal V(\Cal T)$ such that $\mop {res} \Gamma$ is a vertex in $\Cal V_{\Cal T} $ (condition of superficial divergence \cite{SA}, \cite{A.D2000}, \cite{dk98}). Let $\wt {\Cal H}_{\Cal T} = S ( \wt V_{\Cal T})$ be the vector space generated by the specified superficially divergent Feynman graphs of a field theory $\Cal T$. The product is given by disjoint union, the unit $\un$ is identified with the empty graph and the coproduct is defined by:
\begin{eqnarray*}
\Delta (\bar\Gamma ) = \sum_{\substack{\bar\gamma \subseteq \bar\Gamma \\ \bar\Gamma / \bar\gamma \in \Cal T }} \bar \gamma \otimes \bar\Gamma / \bar\gamma,
\end{eqnarray*}
where the sum runs over all locally $1PI$ specified covering subgraphs $ \bar\gamma = (\gamma,\underline{j})$ of $ \bar\Gamma = (\Gamma,\underline{i})$, such that the contracted subgraph $(\Gamma/(\gamma,\underline{j}) , \underline{i})$ is in the theory $\Cal T$.
\begin{remark}
The condition $\bar\Gamma / \bar\gamma \in \Cal T$ is crucial, and means also that $\bar\gamma$ is a ''superficially divergent'' subgraph. For example, in $\varphi^3$:
$$\Gamma = \graphedeltasansindice  \;\;\text{ and } \;\; \gamma = \graphedeltacontratddd \graphev \graphev,$$
gives $\Gamma / \gamma = \graphedeltacontratdddd$ by contraction, which must be eliminated because of the tetravalent vertex.  
\end{remark}
\begin{ex} \label{ex1}
In $\varphi^3$ Theory:
\begin{eqnarray*}
\Delta(\grapheexmdel , 0) &=& (\grapheexmdel , 0) \otimes \graphearetenorcroixzero  + \graphev \graphev \graphev \graphev \otimes (\grapheexmdel , 0)\\
&&\\
&+& (\graphev \graphev \graphecerc , 0)\otimes (\graphecerccroizero , 0)\\
&&\\
&+&  (\graphev \graphev \graphecerc , 1)\otimes (\graphecerccroiun , 0).
\end{eqnarray*}
\textnormal {In QED:}
\begin{eqnarray*}
\Delta (\qedj , 1) &=& \graphevcur \graphevcur \graphevcur \graphevcur\otimes (\qedj , 1)\\
&&\\
&+& (\qedj , 1) \otimes \graphearetenorcroixun  + (\qeda \graphevcur \graphevcur, 0) \otimes (\qedcoiz , 1)\\
&&\\ 
&+&(\qeda \graphevcur \graphevcur, 1) \otimes (\qedcoiu , 1).
\end{eqnarray*}
\end{ex}
\begin{theorem} \cite{DMB}
The coproduct $\Delta$ is coassociative. 
\end{theorem}

The Hopf algebra ${\Cal H}_{\Cal T}$ is given by identifying all graphs of degree zero (the residues) to unit $\un$:
\begin{equation}
{\Cal H}_{\Cal T} = \wt{\Cal H}_{\Cal T} / \Cal J
\end{equation}
where $\Cal J$ is the ideal generated by the elements $\un - \mop{res} \bar\Gamma$ where $\bar\Gamma$ is an $1PI$ specified graph. One immediately checks that $\Cal J$ is a bi-ideal. ${\Cal H}_{\Cal T}$ is a connected graded bialgebra, it is therefore a connected graded Hopf algebra. The coproduct then becomes:
\begin {equation}
\Delta (\bar\Gamma ) = \un \otimes \bar\Gamma + \bar\Gamma \otimes \un + \sum_{\substack{\bar\gamma \hbox{ \sevenrm proper subgraph of}\; \bar\Gamma \\ \hbox{ \sevenrm loc 1PI.}\; \bar\Gamma / \bar\gamma \in \Cal T }} \bar\gamma \otimes \bar\Gamma /\bar\gamma,
\end{equation}

\section{Pre-Lie structure on specified Feynman graphs}
\subsection{Pre-Lie algebras}

\begin{definition}
A Lie algebra over a field $k$ is a vector space $V$ endowed with a bilinear bracket $[., .]$ satisfying:
\begin{enumerate}
\item the antisymmetry:
$$
[x , y] = - [y , x] \hspace{3cm} \forall x, y \in V.
$$
\item the Jacobi identity:
$$
[x , [y , z]] + [y , [z , x]] + [z , [x , y]]  = 0 \hspace{3cm} \forall x, y , z \in V.
$$
\end{enumerate}
\end{definition}

\begin{definition}\cite{ch, dm11}
A left pre-Lie algebra over a field $k$ is a $k$-vector space $\Cal A$ with a binary composition
$\triangleright$ that satisfies the left pre-Lie identity:
\begin{equation}
(x \triangleright y) \triangleright z - x \triangleright (y \triangleright z) = (y \triangleright x) \triangleright z -  y \triangleright (x \triangleright z), 
\end{equation}
for all $x$, $y$, $z \in \Cal A$. Analogously, a right pre-Lie algebra is a $k$-vector space $\Cal A$ with a
binary composition $\triangleleft$ that satisfies the right pre-Lie identity:
\begin{equation}
(x \triangleleft y) \triangleleft z - x \triangleleft (y \triangleleft z) = (x \triangleleft z) \triangleleft y -  x \triangleleft (z \triangleleft y).
\end{equation}
\end{definition}

As any right pre-Lie algebra $(\Cal A, \triangleleft )$ is also a left pre-Lie algebra with product
$x \triangleright y := y \triangleleft x$, we will only consider left pre-Lie algebras for the moment. The left
pre-Lie identity rewrites as:
\begin{equation}
L_{[x,  y]} = [L_x,  L_y],
\end{equation}
where $L_x : A  \longrightarrow A$ is defined by $L_x y = a \triangleright b$, and where the bracket on the
left-hand side is defined by $[a, b] := a \triangleright b -  b  \triangleright a$. As a consequence this bracket
satisfies the Jacobi identity.

\subsection{Pre-Lie algebra of specified graphs}
\begin{definition} Let $\bar \Gamma_1 = (\Gamma_1 , \underline{i})$ and $\bar \Gamma_2 = (\Gamma_2 , \underline{j})$ be two connected specified graphs, i.e $\bar \Gamma_1 = (\Gamma_1 , \underline{i}), \bar \Gamma_2 = (\Gamma_2 , \underline{j}) \in \wt {V}_{\Cal T}$. We defined the insertion of $\bar \Gamma_1$ at $v \in \Cal V(\bar \Gamma_2)$ by: 
\begin{equation}
(\Gamma_1 , \underline{i}) \trianglerightv (\Gamma_2 , \underline{j})  := \left\lbrace
\begin{array}{lcl}
(\Gamma_1 \cupv \Gamma_2 , \underline{j})  \;\;\;\;\; \hbox{if} \;\; v = \mop{res} \bar\Gamma_1  \\
0  \;\; \;\; \;\; \;\; \hfill\hbox{ifnot},
\end{array}\right. 
\end{equation}

where $\Gamma_1 \cupv \Gamma_2$ is a sum of all possibles graphs obtained by replacing the vertex $v$ by the graph $\Gamma_1$ in $\Gamma_2$.\\
We define then the insertion of $\bar \Gamma_1$ in $\bar \Gamma_2$ by:
\begin{equation}
(\Gamma_1 , \underline{i}) \triangleright (\Gamma_2 , \underline{j})  := \sum_{v \in \Cal V (\Gamma_2)} (\Gamma_1 , \underline{i}) \trianglerightv (\Gamma_2 , \underline{j}).
\end{equation} 
\end{definition}
\begin{ex} 
In $\varphi^3$ Theory:
\begin{eqnarray*}
(\graphecerc , 1) \triangleright (\graphecerccroiun , 0) &=& 2 (\grapheexmdel , 0).
\end{eqnarray*}
\textnormal {In QED:}
\begin{eqnarray*}
(\qeda , 1) \triangleright (\qedcoiu , 0) &=& 2 (\qedj , 0).
\end{eqnarray*}
\end{ex}
\begin{theorem} Equiped by $\triangleright$, the space  $\wt {V}_{\Cal T}$ is a pre-Lie algebra. 
\end{theorem}
 \begin{proof} 
Let $(\Gamma_1 , \underline{i}), (\Gamma_2 , \underline{j})$ and $(\Gamma_3 , \underline{k})$ three elements of $\wt {V}_{\Cal T}$, we have:
\begin{eqnarray*}
&&(\Gamma_1 , \underline{i})\triangleright[(\Gamma_2 , \underline{j}) \triangleright (\Gamma_3 , \underline{k})] - [(\Gamma_1 , \underline{i})\triangleright(\Gamma_2 , \underline{j})] \triangleright (\Gamma_3 , \underline{k}) \\
&=& \sum_{v \in \Cal V (\Gamma_3)} (\Gamma_1 , \underline{i})\triangleright [(\Gamma_2 , \underline{j}) \trianglerightv (\Gamma_3 , \underline{k})]- \sum_{s \in \Cal V (\Gamma_2) \;,\;s = \smop{res} \bar\Gamma_1} [(\Gamma_1 , \underline{i})\trianglerights (\Gamma_2 , \underline{j})] \triangleright (\Gamma_3 , \underline{k}) \\ 
&=&\sum_{v \in \Cal V (\Gamma_3);\; v = \smop{res} \bar\Gamma_2} (\Gamma_1 , \underline{i})\triangleright (\Gamma_2 \cupv \Gamma_3 , \underline{k})- \sum_{s \in \Cal V (\Gamma_2)\;,\; s = \smop{res} \bar\Gamma_1} (\Gamma_1 \cups \Gamma_2 , \underline{j})\triangleright (\Gamma_3 , \underline{k}) \\
&=&\sum_{\substack{v \in \Cal V (\Gamma_3)\;,\; s \in \Cal V (\Gamma_2 \cupv \Gamma_3)\\ v = \smop{res} \bar\Gamma_2\;,\;s = \smop{res} \bar\Gamma_1}}(\Gamma_1 , \underline{i})\trianglerights  (\Gamma_2 \cupv \Gamma_3 , \underline{k})- \sum_{\substack{v \in \Cal V (\Gamma_3)\;,\; s \in \Cal V (\Gamma_2)\\ v = \smop{res} \bar\Gamma_2\;,\;s = \smop{res} \bar\Gamma_1}}([\Gamma_1 \cups \Gamma_2 ]\trianglerightv \Gamma_3 , \underline{k})\\
&=&\sum_{\substack{v \in \Cal V (\Gamma_3)\;,\; s \in \Cal V (\Gamma_2 \cupv \Gamma_3)\\ v = \smop{res} \bar\Gamma_2\;,\;s = \smop{res} \bar\Gamma_1}}(\Gamma_1 \cups [\Gamma_2 \cupv \Gamma_3] , \underline{k})- \sum_{\substack{v \in \Cal V (\Gamma_3)\;,\; s \in \Cal V (\Gamma_2)\\ v = \smop{res} \bar\Gamma_2\;,\;s = \smop{res} \bar\Gamma_1}}([\Gamma_1 \cups \Gamma_2 ]\cupv \Gamma_3 , \underline{k})\\
&=&\hspace{-3mm}\sum_{\substack{v \in \Cal V (\Gamma_3)\;,\; s \in \Cal V (\Gamma_2)\\ v = \smop{res} \bar\Gamma_2\;,\;s = \smop{res} \bar\Gamma_1}}([\Gamma_1 \cups \Gamma_2 ]\cupv \Gamma_3 , \underline{k}) + \hspace{-4mm} \sum_{\substack{s, v \in \Cal V (\Gamma_3)\;,\; s \neq v\\ v = \smop{res} \bar\Gamma_2\;,\;s = \smop{res} \bar\Gamma_1}}(\Gamma_1 \cups [\Gamma_2 \cupv \Gamma_3] , \underline{k})- \hspace{-4mm} \sum_{\substack{v \in \Cal V (\Gamma_3)\;,\; s \in \Cal V (\Gamma_2)\\ v = \smop{res} \bar\Gamma_2\;,\;s = \smop{res} \bar\Gamma_1}}([\Gamma_1 \cups \Gamma_2 ]\cupv \Gamma_3 , \underline{k})\\
&=& \sum_{\substack{s, v \in \Cal V (\Gamma_3)\;,\; s \neq v \\ v = \smop{res} \bar\Gamma_2\;,\;s = \smop{res} \bar\Gamma_1}}(\Gamma_1 \cups [\Gamma_2 \cupv \Gamma_3] , \underline{k})\\
&=& \sum_{\substack{s, v \in \Cal V (\Gamma_3)\;,\; s \neq v \\ v = \smop{res} \bar\Gamma_2\;,\;s = \smop{res} \bar\Gamma_1}}(\Gamma_2 \cupv [\Gamma_1 \cups \Gamma_3] , \underline{k})\\
&=&(\Gamma_2 , \underline{i})\triangleright[(\Gamma_1 , \underline{j}) \triangleright (\Gamma_3 , \underline{k})] - [(\Gamma_2 , \underline{i})\triangleright(\Gamma_1 , \underline{j})] \triangleright (\Gamma_3 , \underline{k}),
\end{eqnarray*}
which proves the theorem.
\end{proof}
\begin {remark} We see that the Hopf algebra $\Cal H_{\Cal T}$ is right-sided, i.e, we have:
$$\Delta (\wt {V}_{\Cal T}) \subset \Cal H_{\Cal T} \otimes \wt {V}_{\Cal T}.$$
Then the dual graded of $\wt {V}_{\Cal T}$ is a left pre-Lie algebra \cite{loro}. 
\end{remark}
\begin{corollary} Let be $\bar \Gamma_1 = (\Gamma_1 , \underline{i})$ and $\bar \Gamma_2 = (\Gamma_2 , \underline{j})$ two connected specified graphs, and we define the bracket $\big[\; , \;\big]$ by:
\begin{equation}
[\bar \Gamma_1 , \bar \Gamma_2] := \bar \Gamma_1 \triangleright \bar \Gamma_2 - \bar \Gamma_2  \triangleright  \bar \Gamma_1.
\end{equation}
$\Big(\wt {V}_{\Cal T}, \big[\; , \;\big] \Big)$ is a Lie algebra.
\end{corollary}

\section{Doubling pre-Lie algebra of specified Feynman graphs}
We have studied the concept of doubling bialgebra of specified Feynman graphs \cite[\S 3]{mb} to give a sense to some divergent integrals given by Feynman graphs, We have also studied the notion of doubling bialgebra in the context of rooted trees, we have defined the doubling bialgebras of rooted trees given by extraction contraction and admissible cuts, and we have shown the existence of many relations between these two structures \cite{DB}. In this section we define a pre-Lie structure on the doubling space of connected specified graphs $\wt{V}_{\Cal T}$ noted $\wt{\Cal F}_{\Cal T}$. We prove that $\wt{\Cal F}_{\Cal T}$ is pre-Lie module on $\wt{V}_{\Cal T}$ and we find some relations between the pre-Lie structures on $\wt{\Cal F}_{\Cal T}$ and $\wt{V}_{\Cal T}$.
\subsection {Doubling bialgebra of specified graphs}

\begin{definition} \cite{mb}
Let $\wt{\Cal D}_{\Cal T}$ be the vector space spanned by the pairs $(\bar\Gamma, \bar\gamma)$ of locally $1PI$ specified graphs, with $\bar\gamma \subset \bar\Gamma$ and $\bar\Gamma / \bar\gamma \in \wt {\Cal H}_{\Cal T}$. This is the free commutative algebra generated by the corresponding connected objects. The product $m$ is again given by juxtaposition:
\begin{equation}
m \big((\bar\Gamma_1, \bar\gamma_1)\otimes (\bar\Gamma_2, \bar\gamma_2)\big) = (\bar\Gamma_1  \bar\Gamma_2, \bar\gamma_1 \bar\gamma_2),
\end{equation}
and the coproduct $\Delta$ is defined as follows: 
\begin{equation}
\Delta (\bar\Gamma, \bar\gamma ) = \sum_{\substack{\bar\delta \subseteq \bar\gamma \\ \bar\gamma / \bar\delta \in \Cal T }} ( \bar \Gamma, \bar\delta )  \otimes ( \bar\Gamma / \bar\delta, \bar\gamma / \bar\delta)
\end{equation}
\end{definition}
\begin{proposition} \cite{mb}
$({\wt{\Cal D}}_{\Cal T}, m, \Delta, u,\varepsilon)$ is a graded bialgebra, and 
\begin{eqnarray*}
P_2: {\wt{\Cal D}}_{\Cal T} &\longrightarrow& {\wt{\Cal H}}_{\Cal T} \\ (\bar\Gamma, \bar\gamma ) &\longmapsto& \bar\gamma
\end{eqnarray*} 
is a bialgebra morphism. 
\end{proposition}

\subsection{Doubling pre-Lie algebra of specified  graphs}
Let $\wt{\Cal F}_{\Cal T}$ be the vector space spanned by the pairs $(\bar\Gamma, \bar\gamma)$ where $\Gamma$ is a $1PI$ connected specified graph, $\bar\gamma \subset \bar\Gamma$ and $\bar\Gamma / \bar\gamma \in \wt {V}_{\Cal T}$.

\begin{definition} Let be $(\bar \Gamma_1 , \bar \gamma_1)$ and $(\bar \Gamma_2, \bar \gamma_2)$ two elements of $\wt{\Cal F}_{\Cal T}$, we defined the map $\odot$ by:
\begin{equation}
(\bar \Gamma_1 , \bar \gamma_1)  \odot (\bar \Gamma_2, \bar \gamma_2) := \sum_{\substack{v \in \Cal V(\Gamma_2 - \gamma_2)\; , v = \smop{res} \Gamma_1}}  (\bar\Gamma_1 \trianglerightv \bar\Gamma_2 , \bar\gamma_1  \bar\gamma_2),
\end{equation}
where the notation $v \in \Cal V(\Gamma_2 - \gamma_2)$ denotes that $v$ is a vertex of $\Gamma_2$ but is not a vertex of $\gamma_2$.
\end{definition} 
\begin{ex} 
\textnormal {In QED:}
\begin{eqnarray*}
\left((\grapheexmdel , 0),  (\graphecerc , 1) \right) &\odot& \left((\grdtz , 1)  , (\qedcoiz , 0) \right)\\
&=& 2 \left((\grpddz , 1) , (\graphecerc , 1)(\qedcoiz , 0) \right)  \\
&& \; + \; 2  \left( (\grpddze , 1), (\graphecerc , 1)(\qedcoiz , 0)\right)
\end{eqnarray*}
\end{ex}
\begin{theorem} Equiped by $\odot$, the space $\wt{\Cal F}_{\Cal T}$ is a pre-Lie algebra. 
\end{theorem}
\begin{proof} 
Let $(\bar \Gamma_1 , \bar \gamma_1), (\bar \Gamma_2 , \bar \gamma_2)$ and $(\bar \Gamma_3 , \bar \gamma_3)$ be three elements of $\wt {\Cal F}_{\Cal T}$, we have:
 \begin{eqnarray*}
(\bar \Gamma_1 , \bar \gamma_1) \odot \big[(\bar \Gamma_2 , \bar \gamma_2) \odot(\bar \Gamma_3 , \bar \gamma_3)\big]  &=& 
(\bar \Gamma_1 , \bar \gamma_1) \odot \big(\sum_{\substack{v \in \Cal V(\Gamma_3 - \gamma_3)\; , v = \smop{res} \Gamma_2}}  (\bar\Gamma_2 \trianglerightv \bar\Gamma_3 , \bar\gamma_2  \bar\gamma_3)\big)\\
&=& \sum_{\substack{v \in \Cal V(\Gamma_3 - \gamma_3)\; r \in \Cal V(\Gamma_2 \trianglerightv \Gamma_3 -\bar\gamma_2 \bar\gamma_3)  \\ v = \smop{res} \Gamma_2 \;,\;r = \smop{res} \Gamma_1}}  \big( \bar \Gamma_1 \trianglerightr (\bar\Gamma_2 \trianglerightv \bar\Gamma_3) , \bar\gamma_1 \bar\gamma_2  \bar\gamma_3\big)\\
&=& \sum_{\substack{r\;,\;v \in \Cal V(\Gamma_3 - \gamma_3)\; r \neq v  \\ v = \smop{res} \Gamma_2 \;,\;r = \smop{res} \Gamma_1}}  \big( \bar \Gamma_1 \trianglerightr (\bar\Gamma_2 \trianglerightv \bar\Gamma_3) , \bar\gamma_1 \bar\gamma_2  \bar\gamma_3\big)\\
&& \;+\;   \sum_{\substack{v \in \Cal V(\Gamma_3 - \gamma_3)\; r \in \Cal V(\Gamma_2 -\bar\gamma_2)  \\ v = \smop{res} \Gamma_2 \;,\;r = \smop{res} \Gamma_1}}  \big( \bar \Gamma_1 \trianglerightr (\bar\Gamma_2 \trianglerightv \bar\Gamma_3) , \bar\gamma_1 \bar\gamma_2  \bar\gamma_3\big).
\end{eqnarray*}
On the other hand we have: 
 \begin{eqnarray*}
\big[(\bar \Gamma_1 , \bar \gamma_1) \odot (\bar \Gamma_2 , \bar \gamma_2) \big]\odot(\bar \Gamma_3 , \bar \gamma_3)  &=& \sum_{\substack{r \in \Cal V(\Gamma_2 - \gamma_2)\; , r = \smop{res} \Gamma_1}} (\bar\Gamma_1 \trianglerightr \bar\Gamma_2 , \bar\gamma_1  \bar\gamma_2)  \odot(\bar \Gamma_3 , \bar \gamma_3)\\
&=& \sum_{\substack{v \in \Cal V(\Gamma_3 - \gamma_3)\; r \in \Cal V(\Gamma_2 -\bar\gamma_2 )  \\ v = \smop{res} \bar\Gamma_1 \trianglerightr \bar\Gamma_2 \;,\;r = \smop{res} \Gamma_1}} \big( (\bar\Gamma_1 \trianglerightr \bar\Gamma_2) \trianglerightv \bar\Gamma_3 , \bar\gamma_1  \bar\gamma_2  \bar\gamma_2 \big) \\
&=& \sum_{\substack{v \in \Cal V(\Gamma_3 - \gamma_3)\; r \in \Cal V(\Gamma_2 -\bar\gamma_2 )  \\ v = \smop{res} \bar\Gamma_2 \;,\;r = \smop{res} \Gamma_1}} \big( (\bar\Gamma_1 \trianglerightr \bar\Gamma_2) \trianglerightv \bar\Gamma_3 , \bar\gamma_1  \bar\gamma_2   \bar\gamma_3 \big).
\end{eqnarray*}
Then we have:
\begin{eqnarray*}
(\bar \Gamma_1 , \bar \gamma_1) \odot \big[(\bar \Gamma_2 , \bar \gamma_2) \odot(\bar \Gamma_3 , \bar \gamma_3)\big] &-& \big[(\bar \Gamma_1 , \bar \gamma_1) \odot (\bar \Gamma_2 , \bar \gamma_2) \big]\odot(\bar \Gamma_3 , \bar \gamma_3)\\
  &=& \sum_{\substack{r\;,\;v \in \Cal V(\Gamma_3 - \gamma_3)\; r \neq v  \\ v = \smop{res} \Gamma_2 \;,\;r = \smop{res} \Gamma_1}}  \big( \bar \Gamma_1 \trianglerightr (\bar\Gamma_2 \trianglerightv \bar\Gamma_3) , \gamma_1\cup \bar\gamma_2   \bar\gamma_3\big)\\
&& \;+\;  \sum_{\substack{v \in \Cal V(\Gamma_3 - \gamma_3)\; r \in \Cal V(\Gamma_2 -\bar\gamma_2)  \\ v = \smop{res} \Gamma_2 \;,\;r = \smop{res} \Gamma_1}}  \big( \bar \Gamma_1 \trianglerightr (\bar\Gamma_2 \trianglerightv \bar\Gamma_3) , \gamma_1\cup \bar\gamma_2   \bar\gamma_3\big)\\
	&& \;-\;  \sum_{\substack{v \in \Cal V(\Gamma_3 - \gamma_3)\; r \in \Cal V(\Gamma_2 -\bar\gamma_2 )  \\ v = \smop{res} \bar\Gamma_2 \;,\;r = \smop{res} \Gamma_1}} \big( (\bar\Gamma_1 \trianglerightr \bar\Gamma_2) \trianglerightv \bar\Gamma_3 , \bar\gamma_1  \bar\gamma_2   \bar\gamma_3 \big)\\
	&=& \sum_{\substack{r\;,\;v \in \Cal V(\Gamma_3 - \gamma_3)\; r \neq v  \\ v = \smop{res} \Gamma_2 \;,\;r = \smop{res} \Gamma_1}}  \big( \bar \Gamma_1 \trianglerightr (\bar\Gamma_2 \trianglerightv \bar\Gamma_3) , \gamma_1 \bar\gamma_2   \bar\gamma_3\big)\\
	&=&\\ (\bar \Gamma_2 , \bar \gamma_2) \odot \big[(\bar \Gamma_1 , \bar \gamma_1) \odot(\bar \Gamma_3 , \bar \gamma_3)\big] &-& \big[(\bar \Gamma_2 , \bar \gamma_2) \odot (\bar \Gamma_1 , \bar \gamma_1) \big]\odot(\bar \Gamma_3 , \bar \gamma_3).
\end{eqnarray*}
Consequently, $\odot$ is pre-Lie. 
\end{proof}
\vspace{0.2cm}
\subsection{Pre-Lie module}
\begin{definition}
Let $(\Cal A, \circ)$ be pre-Lie algebra. A left $\Cal A$-module is a Vector space $\Cal M$ provided with a bilinear law noted $\succ : \Cal A \otimes \Cal M \longrightarrow \Cal M$ such that for all $a, b \in \Cal A$ and $m \in \Cal M$, we have:
\begin{equation}
   a \succ (b \succ m) - (a \circ b)\succ m = b \succ (a \succ m) - (b \circ a) \succ m. 
\end{equation}
\end{definition}
\begin{definition} 
Let $\bar \Gamma_1 \in \wt{V}_{\Cal T}$ and $(\bar \Gamma_2, \bar \gamma_2) \in \wt{\Cal F}_{\Cal T}$, we define the map $\rightarrow$ by:
\begin{equation}
\bar \Gamma_1 \rightarrow (\bar \Gamma_2, \bar \gamma_2) := \sum_{\substack{v \in \Cal V(\gamma_2)\; , v = \smop{res} \Gamma_1}}  (\bar\Gamma_1 \trianglerightv \bar\Gamma_2 , \bar\Gamma_1 \trianglerightv \bar\gamma_2). 
\end{equation} 
\end{definition}
\begin{theorem} Equiped by $\rightarrow$, the space $\wt{\Cal F}_{\Cal T}$ is a pre-Lie module on  $\wt{V}_{\Cal T}$. In other words for any $\bar \Gamma_1, \bar \Gamma_2 \in \wt {V}_{\Cal T}$ and $(\bar \Gamma_3 , \bar \gamma_3) \in \wt {\Cal F}_{\Cal T}$, we have:
$$
\bar \Gamma_1 \rightarrow \Big[\bar \Gamma_2 \rightarrow(\bar \Gamma_3 , \bar \gamma_3)\Big] - (\bar \Gamma_1  \triangleright \bar \Gamma_2) \rightarrow(\bar \Gamma_3 , \bar \gamma_3) = \bar \Gamma_2 \rightarrow \Big[\bar \Gamma_1 \rightarrow(\bar \Gamma_3 , \bar \gamma_3)\Big] - (\bar \Gamma_2  \triangleright \bar \Gamma_1) \rightarrow(\bar \Gamma_3 , \bar \gamma_3).$$ 
\end{theorem}
\begin{proof} 
Let $\bar \Gamma_1, \bar \Gamma_2$ be two elements of $\wt {V}_{\Cal T}$ and $(\bar \Gamma_3 , \bar \gamma_3)$ an element of $\wt {\Cal F}_{\Cal T}$, we have:
 \begin{eqnarray*}
\bar \Gamma_1 \rightarrow \Big[\bar \Gamma_2 \rightarrow(\bar \Gamma_3 , \bar \gamma_3)\Big]  &=& 
\bar \Gamma_1 \rightarrow \Big[\sum_{\substack{s \in \Cal V(\gamma_3)\; , v = \smop{res} \Gamma_2}}  (\bar\Gamma_2 \trianglerights \bar\Gamma_3 , \bar\Gamma_2 \trianglerights \bar\gamma_3)\Big]\\
&=& \sum_{\substack{s \in \Cal V(\gamma_3)\; l \in \Cal V(\Gamma_2 \trianglerights \gamma_3)  \\ s = \smop{res} \Gamma_2 \;,\;l = \smop{res} \Gamma_1}}  [ \bar \Gamma_1 \trianglerightl (\bar\Gamma_2 \trianglerights \bar\Gamma_3) , \bar \Gamma_1 \trianglerightl (\bar\Gamma_2 \trianglerights \bar\gamma_3)]\\
&=& \sum_{\substack{l\;, \; s \in \Cal V(\gamma_3)\; v \neq s  \\ s = \smop{res} \Gamma_2 \;,\;l = \smop{res} \Gamma_1}}  [ \bar \Gamma_1 \trianglerightl (\bar\Gamma_2 \trianglerights \bar\Gamma_3) , \bar \Gamma_1 \trianglerightl (\bar\Gamma_2 \trianglerights \bar\gamma_3)]   \\
&& \;+\; \sum_{\substack{l \in \Cal V(\Gamma_2)\;, \; s \in \Cal V(\gamma_3)\\ s = \smop{res} \Gamma_2 \;,\;l = \smop{res} \Gamma_1}}  [ \bar \Gamma_1 \trianglerightl (\bar\Gamma_2 \trianglerights \bar\Gamma_3) , \bar \Gamma_1 \trianglerightl (\bar\Gamma_2 \trianglerights \bar\gamma_3)].
\end{eqnarray*}
On the other hand we have:
 \begin{eqnarray*}
(\bar \Gamma_1  \triangleright \bar \Gamma_2) \rightarrow(\bar \Gamma_3 , \bar \gamma_3)  &=& 
\sum_{\substack{l \in \Cal V(\gamma_2)\; , l = \smop{res} \Gamma_1}}\bar\Gamma_1 \trianglerightl \bar\Gamma_2 \rightarrow(\bar \Gamma_3 , \bar \gamma_3)\\
&=& 
\sum_{\substack{l \in \Cal V(\Gamma_2)\; s \in \Cal V(\gamma_3)  \\ l = \smop{res} \Gamma_1 \;,\;s = \smop{res} \bar\Gamma_1 \trianglerightl \bar\Gamma_2}} \big[(\bar\Gamma_1 \trianglerightl \bar\Gamma_2)\trianglerights \bar\Gamma_3   , (\bar\Gamma_1 \trianglerightl \bar\Gamma_2) \trianglerights \bar\gamma_3\big]\\
&=& 
\sum_{\substack{l \in \Cal V(\Gamma_2)\; s \in \Cal V(\gamma_3)  \\ l = \smop{res} \Gamma_1 \;,\;s = \smop{res}\bar\Gamma_2}} \big[(\bar\Gamma_1 \trianglerightl \bar\Gamma_2)\trianglerights \bar\Gamma_3   , (\bar\Gamma_1 \trianglerightl \bar\Gamma_2) \trianglerights \bar\gamma_3\big].
\end{eqnarray*}
Then, we have:
\begin{eqnarray*}
\bar \Gamma_1  \rightarrow \Big[\bar \Gamma_2  \rightarrow(\bar \Gamma_3 , \bar \gamma_3)\Big] &-& \Big[\bar \Gamma_1  \triangleright\bar \Gamma_2 \Big] \rightarrow(\bar \Gamma_3 , \bar \gamma_3)\\ &=& \sum_{\substack{l\;, \; s \in \Cal V(\gamma_3)\; v \neq s  \\ s = \smop{res} \Gamma_2 \;,\;l = \smop{res} \Gamma_1}}  [ \bar \Gamma_1 \trianglerightl (\bar\Gamma_2 \trianglerights \bar\Gamma_3) , \bar \Gamma_1 \trianglerightl (\bar\Gamma_2 \trianglerights \bar\gamma_3)]   \\
&& \;+\;\sum_{\substack{l \in \Cal V(\Gamma_2)\;, \; s \in \Cal V(\gamma_3)\\ s = \smop{res} \Gamma_2 \;,\;l = \smop{res} \Gamma_1}}  [ \bar \Gamma_1 \trianglerightl (\bar\Gamma_2 \trianglerights \bar\Gamma_3) , \bar \Gamma_1 \trianglerightl (\bar\Gamma_2 \trianglerights \bar\gamma_3)]\\
&& \;-\;\sum_{\substack{l \in \Cal V(\Gamma_2)\; s \in \Cal V(\gamma_3)  \\ l = \smop{res} \Gamma_1 \;,\;s = \smop{res}\bar\Gamma_2}} \big[(\bar\Gamma_1 \trianglerightl \bar\Gamma_2)\trianglerights \bar\Gamma_3   , (\bar\Gamma_1 \trianglerightl \bar\Gamma_2) \trianglerights \bar\gamma_3\big]\\
&=&\sum_{\substack{l\;, \; s \in \Cal V(\gamma_3)\; v \neq s  \\ s = \smop{res} \Gamma_2 \;,\;l = \smop{res} \Gamma_1}}  [ \bar \Gamma_1 \trianglerightl (\bar\Gamma_2 \trianglerights \bar\Gamma_3) , \bar \Gamma_1 \trianglerightl (\bar\Gamma_2 \trianglerights \bar\gamma_3)],
\end{eqnarray*} 
which is symmetric in $\bar\Gamma_1$ and $\bar\Gamma_2$. Therefore:
$$
\bar \Gamma_1 \rightarrow \Big[\bar \Gamma_2 \rightarrow(\bar \Gamma_3 , \bar \gamma_3)\Big] - (\bar \Gamma_1  \triangleright \bar \Gamma_2) \rightarrow(\bar \Gamma_3 , \bar \gamma_3) = \bar \Gamma_2 \rightarrow \Big[\bar \Gamma_1 \rightarrow(\bar \Gamma_3 , \bar \gamma_3)\Big] - (\bar \Gamma_2  \triangleright \bar \Gamma_1) \rightarrow(\bar \Gamma_3 , \bar \gamma_3),$$
which proves the theorem.
\end{proof}

\subsection{Relation between $\rightarrow$ and $\odot$}
In this section , we prove that there exist relations between the action $\rightarrow$ and the pre-Lie product $\odot$ defined on $\wt{\Cal F}_{\Cal T}$. 

\begin{theorem} The law $\rightarrow$ is a derivation of the algebra $(\wt{\Cal F}_{\Cal T}, \odot)$. In other words, for any $\bar \Gamma_1 \in \wt {V}_{\Cal T}$ and  $(\bar \Gamma_2, \bar \gamma_2), (\bar \Gamma_3, \bar \gamma_3) \in \wt {\Cal F}_{\Cal T}$, we have:
\begin{equation*}
\bar \Gamma_1 \rightarrow  \big((\bar \Gamma_2, \bar \gamma_2) \odot (\bar \Gamma_3, \bar \gamma_3) \big) = \big(\bar \Gamma_1 \rightarrow (\bar \Gamma_2, \bar \gamma_2) \big) \odot (\bar \Gamma_3, \bar \gamma_3) + (\bar \Gamma_2, \bar \gamma_2) \odot \big(\bar \Gamma_1\rightarrow (\bar \Gamma_3, \bar \gamma_3) \big).
\end{equation*}
\end{theorem}
\begin{proof}
Let be $\bar \Gamma_1 \in \wt {V}_{\Cal T}$ and  $(\bar \Gamma_2, \bar \gamma_2), (\bar \Gamma_3, \bar \gamma_3)\in \wt {\Cal F}_{\Cal T}$, we have:
\begin{eqnarray*}
\bar \Gamma_1 \rightarrow  \big((\bar \Gamma_2, \bar \gamma_2) \odot (\bar \Gamma_3, \bar \gamma_3) \big) &=& 
\bar \Gamma_1 \rightarrow \big( \sum_{\substack{v \in \Cal V(\Gamma_3 - \gamma_3)\\ v = \smop{res} \bar\Gamma_2 }} (\bar\Gamma_2  \trianglerightv \bar \Gamma_3, \gamma_2 \gamma_3) \big)\\
&=&\sum_{\substack{s \in  \Cal V(\gamma_2 \gamma_3)\;\; v \in \Cal V(\Gamma_3 - \gamma_3)\\ s = \smop{res} \bar\Gamma_1 \;\; v = \smop{res} \bar\Gamma_2 }} \big( \bar\Gamma_1  \trianglerights(\bar\Gamma_2  \trianglerightv \bar \Gamma_3) , \bar\Gamma_1  \trianglerights(\gamma_2 \gamma_3) \big)\\ 
&=&\sum_{\substack{s \in  \Cal V(\gamma_2)\;\; v \in \Cal V(\Gamma_3 - \gamma_3)\\ s = \smop{res} \bar\Gamma_1 \;\; v = \smop{res} \bar\Gamma_2 }} \big( (\bar\Gamma_1  \trianglerights\bar\Gamma_2)  \trianglerightv \bar \Gamma_3 , (\bar\Gamma_1  \trianglerights\gamma_2) \gamma_3 \big)\\ 
&+&\sum_{\substack{s \in  \Cal V(\gamma_3)\;\; v \in \Cal V(\Gamma_3 - \gamma_3)\\ s = \smop{res} \bar\Gamma_1 \;\; v = \smop{res} \bar\Gamma_2 }} \big(\bar\Gamma_2  \trianglerightv (\bar\Gamma_1  \trianglerights \bar \Gamma_3) , \bar\gamma_2  (\Gamma_1\trianglerights \gamma_3) \big)\\ 
&=&
\big(\bar \Gamma_1 \rightarrow (\bar \Gamma_2, \bar \gamma_2) \big) \odot (\bar \Gamma_3, \bar \gamma_3) + (\bar \Gamma_2, \bar \gamma_2) \odot \big(\bar \Gamma_1\rightarrow (\bar \Gamma_3, \bar \gamma_3) \big).
\end{eqnarray*}
\end{proof}

\begin{theorem}
The following diagram is commutative:
\diagramme{
\xymatrix{
\wt{V}_{\Cal T} \otimes \wt{\Cal F}_{\Cal T} \ar[d]_{I\otimes P_2} \ar[rrr]^{\rightarrow} 
&&&\wt{\Cal F}_{\Cal T} \ar[d]^{P_2}\\
\wt{V}_{\Cal T}\otimes\wt{\Cal H}_{\Cal T} \ar[rrr]_{\triangleright} 
&&& \wt{\Cal H}_{\Cal T} }
}
In other words, the projection on the second component $P_2$ is a morphism of pre-Lie modules.
\end{theorem}
\begin{proof}
Let be $\bar \Gamma_1 \in \wt {V}_{\Cal T}$ and $(\bar \Gamma_2, \bar \gamma_2) \in \wt {\Cal F}_{\Cal T}$, we have:
\begin{eqnarray*}
P_2 \big(\bar \Gamma_1 \rightarrow (\bar \Gamma_2, \bar \gamma_2)\big) &=& P_2 \big(\sum_{\substack{v \in \Cal V(\gamma_2)\\ v = \smop{res} \bar\gamma_1 }} (\bar\Gamma_1  \trianglerightv \bar \Gamma_2, \bar\Gamma_1 \trianglerightv  \gamma_2) \big)\\
&=&  \sum_{\substack{v \in \Cal V(\gamma_2)\\ v = \smop{res} \bar\Gamma_1 }}  \bar\Gamma_1  \trianglerightv \bar \gamma_2\\
&=&  \bar\Gamma_1 \triangleright  \bar\gamma_2 \\
&=&  \bar\Gamma_1 \triangleright  P_2 (\bar\Gamma_2, \bar\gamma_2)\\
&=&  (I \otimes P_2) \big(\bar\Gamma_1 \otimes(\bar\Gamma_2, \bar\gamma_2)\big),
\end{eqnarray*}
which proves the theorem.
\end{proof}
\section{Enveloping algebra of pre-Lie algebra}
In this section We use the method of  Oudom and Guin  \cite{og} to describe the enveloping algebra of the pre-Lie algebra.
\begin{definition}\label{odam}\cite{og}
Let $(A, \triangleright)$ be a pre-Lie algebra. We consider the Hopf symmetric algebra $\Cal S (A)$ equipped with its usual coproduct $\Delta$. We extend the product $\triangleright$ to $\Cal S (A)$. Let $a, b$ and $c \in \Cal S (A)$, and $x \in A$. We put:
\begin{eqnarray*}
\un \triangleright a &=& a\\
a \triangleright \un &=& \varepsilon(a) \un\\
(x a) \triangleright b &=& x \triangleright (a \triangleright b) - (x \triangleright a) \triangleright b\\
a \triangleright (b  c) &=& \sum_{a}  (a^{(1)} \triangleright b) (a^{(2)} \triangleright c).
\end{eqnarray*}

On $\Cal S (A)$, we define a product $\star$ by:
\begin{eqnarray*}
a \star b &=& \sum_{a}  a^{(1)} (a^{(2)} \triangleright b).
\end{eqnarray*}
\end{definition}
\begin{theorem} The space $(\Cal S (A), \star, \Delta)$is a Hopf algebra which is isomorphic to the enveloping Hopf algebra $\Cal U(A_{Lie})$ of the Lie algebra $A_{Lie}$. 
\end{theorem}
\begin{proof} This theorem was proved by Oudom and Guin in \cite{og}. 
\end{proof}
\subsection{Enveloping algebra of the pre-Lie algebra of specified Feynman graphs}
We consider the Hopf symmetric algebra $\wt {\Cal H'}_{\Cal T} : = \Cal S (\wt {V}_{\Cal T})$ of the pre-Lie algebra $(\wt {V}_{\Cal T}, \triangleright)$, equipped with its usual unshuffling coproduct $\Psi$. We extend the product $\triangleright$ to $\wt {\Cal H'}$ by the same method used in Defintion \ref{odam} and we define a product $\star$ on $\wt {\Cal H'}_{\Cal T}$ by:
\begin{eqnarray*}
\bar\Gamma \star \bar\Gamma' &=& \sum_{(\Gamma)}  \bar\Gamma^{(1)} (\bar\Gamma^{(2)} \triangleright \bar\Gamma').
\end{eqnarray*}
By construction, the space $(\wt {\Cal H'}_{\Cal T}, \star, \Psi)$ is a Hopf algebra. 
\begin{ex} 
\textnormal {In QED:}
\begin{eqnarray*}
(\graphecerc , 1) \triangleright \big((\qedcoiu , 0) (\graphecerccroiun , 0)\big) &=&   2 (\oudam, 0) (\graphecerccroiun , 0)\\
&+& 2(\qedcoiu , 0) (\grapheexmdel , 0).\\
\big((\graphecerc , 1) (\qeda , 0) \big) \triangleright (\prelied , 0) &=&  4(\preliekb, 0).
\end{eqnarray*}

\begin{eqnarray*}
\big((\graphecerc , 1) (\qeda , 0) \big)\star (\prelied , 0) &=&  \un \triangleright (\prelied , 0) (\graphecerc , 1) (\qeda , 0) \\
&+&  (\qeda , 0) \triangleright (\prelied , 0) (\graphecerc , 1) \\
&+& (\graphecerc , 1) \triangleright (\prelied , 0) (\qeda , 0) \\
&+& \big((\graphecerc , 1) (\qeda , 0) \big) \triangleright (\prelied , 0)\\
&=& (\prelied , 0) (\graphecerc , 1) (\qeda , 0)\\
&+&  4 (\preliekb, 0)   \\
&+& 2 (\prelz , 0) (\graphecerc , 1) + 2 (\prelieu, 0) (\qeda , 0).
\end{eqnarray*}
\end{ex}
\begin{theorem} $(\wt {\Cal H'}_{\Cal T}, m, \Psi)$ is a comodule-coalgebra on $(\wt {\Cal H}_{\Cal T},  m, \Delta)$, where $\wt {\Cal H}_{\Cal T}$ and $\wt {\Cal H'}_{\Cal T}$ are isomorphic as algebras.
\end{theorem}
\begin{proof} 
It is clear that $\Delta : \wt {\Cal H'}_{\Cal T}\longrightarrow \wt {\Cal H}_{\Cal T} \otimes \wt {\Cal H'}_{\Cal T}$ is a coaction, that means that $\Delta$ is coassocitive. Let $(\Gamma , \underline{i}) \in \wt {\Cal H'}_{\Cal T}$, we have:
 \begin{eqnarray*}
(I \otimes \Delta) \circ \Delta (\bar \Gamma) &=& (I \otimes \Delta)(\sum_{\gamma \subset \Gamma} \bar\gamma \otimes \bar\Gamma/\bar\gamma)\\ 
&=&\sum_{\delta \subset \gamma \subset \Gamma} \bar\delta \otimes \bar\gamma/\bar\delta \otimes \bar\Gamma/\bar\gamma\\
&=& (\Delta\otimes I) \circ \Delta (\bar \Gamma).
\end{eqnarray*}
Second, we prove that the coproduct $\Psi$ is morphism of left  $\wt {\Cal H'}_{\Cal T}$-comodules. This amounts to the commutativity of the following diagram:
\diagramme{
\xymatrix{
 \wt {\Cal H'}_{\Cal T} \ar[d]_{\Psi} \ar[rr]^{\Delta} 
&& \wt {\Cal H}_{\Cal T} \otimes  \wt {\Cal H'}_{\Cal T} \ar[d]^{I \otimes \Psi}\\
  \wt {\Cal H'}_{\Cal T} \otimes  \wt {\Cal H'}_{\Cal T} \ar[d]_{\Delta\otimes \Delta}&& \wt {\Cal H}_{\Cal T} \otimes  \wt {\Cal H'}_{\Cal T} \otimes  \wt {\Cal H'}_{\Cal T} \\
 \wt {\Cal H}_{\Cal T} \otimes  \wt {\Cal H'}_{\Cal T} \otimes  \wt {\Cal H}_{\Cal T} \otimes  \wt {\Cal H'}_{\Cal T} \ar[rr]_{\tau^{23}} 
&&\ar[u]_{m \otimes I }  \wt {\Cal H}_{\Cal T} \otimes  \wt {\Cal H}_{\Cal T} \otimes   \wt {\Cal H'}_{\Cal T} \otimes  \wt {\Cal H'}_{\Cal T}
 }
}
We use the shorthand notation: $m^{13} := (m \otimes I)\circ \tau^{23}$.  
\begin{eqnarray*}
(I \otimes \Psi) \circ \Delta (\bar \Gamma) &=& (I \otimes \Psi)(\sum_{\gamma \subset \Gamma} \bar\gamma \otimes \bar\Gamma/\bar\gamma)\\ 
&=&\sum_{\gamma \subset \Gamma} \bar\gamma \otimes \Psi(\bar\Gamma/\bar\gamma)\\
&=&\sum_{{\gamma \subset \Gamma},{(\Gamma/\gamma)}} \bar\gamma \otimes (\bar\Gamma/ \bar\gamma)^{(1)} \otimes (\bar\Gamma/\bar\gamma)^{(2)}\\
&=& \sum_{{\gamma_1 \subset \Gamma^{(1)}},{\gamma_2 \subset \Gamma^{(2)}}} \bar\gamma_1 \bar\gamma_2 \otimes \bar\Gamma^{(1)}/ \bar\gamma_1 \otimes \bar\Gamma^{(2)}/\bar\gamma_2.
\end{eqnarray*}
In the last passage, we used the fact that a graph and any of its contracted have the same number of connected components. 
\begin{eqnarray*}
(m^{13} \otimes I) \circ (\Delta \otimes \Delta)\circ \Psi (\bar \Gamma) &=& (m^{13} \otimes I)\big( \sum_{(\Gamma)} \Delta(\bar\Gamma^{(1)}) \otimes  \Delta (\bar\Gamma^{(2)})\big) \\ 
&=& (m^{13} \otimes I) \big(\sum_{{\gamma_1 \subset \Gamma^{(1)}},{\gamma_2 \subset \Gamma^{(2)}}} \bar\gamma_1\otimes \bar\gamma_2 \otimes \bar\Gamma^{(1)}/ \bar\gamma_1 \otimes \bar\Gamma^{(2)}/\bar\gamma_2\big)\\
&=&  \sum_{{\gamma_1 \subset \Gamma^{(1)}},{\gamma_2 \subset \Gamma^{(2)}}} \bar\gamma_1 \bar\gamma_2 \otimes \bar\Gamma^{(1)}/ \bar\gamma_1 \otimes \bar\Gamma^{(2)}/\bar\gamma_2,
\end{eqnarray*}
which proves the theorem.
\end{proof}
\subsection{Enveloping algebra of doubling pre-Lie algebra of specified  graphs}
Let $\wt{\Cal F}_{\Cal T}$ be the vector space spanned by the pairs $(\bar\Gamma, \bar\gamma)$ where $\Gamma$ is an $1PI$ connected specified graph, $\bar\gamma \subset \bar\Gamma$ and $\bar\Gamma / \bar\gamma \in \wt {V}_{\Cal T}$. We use the method of  Oudom and Guin  \cite{og} to find the enveloping doubling pre-Lie algebra of specified Feynman graphs. We have $(\wt {\Cal F}_{\Cal T}, \odot)$ is a pre-Lie algebra,so we consider the Hopf symmetric algebra $\wt {\Cal D'}_{\Cal T} : = \Cal S (\wt {\Cal F}_{\Cal T})$ equipped with its usual unshuffling coproduct $\Phi$. We extend the product $\odot$ to $\wt {\Cal D'}_{\Cal T}$ by using Definition \ref{odam} and we define a product $\bigstar$ on $\wt {\Cal D'}_{\Cal T}$ by: 
\begin{eqnarray*}
(\bar\Gamma, \bar\gamma) \bigstar (\bar\Gamma', \bar\gamma') &=& \sum_{(\Gamma, \gamma)}  (\bar\Gamma, \bar\gamma)^{(1)} \big((\bar\Gamma, \bar\gamma)^{(2)} \odot (\Gamma', \gamma')\big).
\end{eqnarray*}

By construction, the space $(\wt {\Cal D'}_{\Cal T}, \bigstar, \Phi)$ is a Hopf algebra. 

\begin{theorem} $(\wt {\Cal D'}_{\Cal T}, m, \Phi)$ is a comodule-coalgebra on $(\wt {\Cal D}_{\Cal T},  m, \Delta)$, where $\wt {\Cal D}_{\Cal T}$ and $\wt {\Cal D'}_{\Cal T}$ are isomorphic as algebras.
\end{theorem}
\begin{proof} 
It is clear taht $\Delta : \wt {\Cal D'}_{\Cal T} \longrightarrow \wt {\Cal D}_{\Cal T} \otimes  \wt {\Cal D'}_{\Cal T}$ is a coaction, that means that $\Delta$ is coassocitive. 

Second, we prove that the coproduct $\Phi$ is morphism of left  $\wt {\Cal D'}_{\Cal T}$-comodules. This amounts to the commutativity of the following diagram:
\diagramme{
\xymatrix{
 \wt {\Cal D'}_{\Cal T} \ar[d]_{\Phi} \ar[rr]^{\Delta} 
&& \wt {\Cal D}_{\Cal T} \otimes  \wt {\Cal D'}_{\Cal T} \ar[d]^{I \otimes \Phi}\\
  \wt {\Cal D'}_{\Cal T} \otimes  \wt {\Cal D'}_{\Cal T} \ar[d]_{\Delta\otimes \Delta}&& \wt {\Cal D}_{\Cal T} \otimes  \wt {\Cal D'}_{\Cal T} \otimes  \wt {\Cal D'}_{\Cal T} \\
 \wt {\Cal D}_{\Cal T} \otimes  \wt {\Cal D'}_{\Cal T} \otimes  \wt {\Cal D}_{\Cal T} \otimes  \wt {\Cal D'}_{\Cal T} \ar[rr]_{\tau^{23}} 
&&\ar[u]_{m \otimes I}  \wt {\Cal D}_{\Cal T} \otimes  \wt {\Cal D}_{\Cal T} \otimes   \wt {\Cal D'}_{\Cal T} \otimes  \wt {\Cal D'}_{\Cal T}
 }
}
\begin{eqnarray*}
(I \otimes \Phi) \circ \Delta (\bar \Gamma, \bar \gamma) &=& (I \otimes \Phi) \left(\sum_{\substack{\bar\delta \subseteq \bar\gamma }} ( \bar \Gamma, \bar\delta )  \otimes ( \bar\Gamma / \bar\delta, \bar\gamma / \bar\delta) \right)\\ 
&=& \sum_{\substack{\bar\delta \subseteq \bar\gamma }} ( \bar \Gamma, \bar\delta )  \otimes \phi( \bar\Gamma / \bar\delta, \bar\gamma / \bar\delta)\\
&=& \sum_{\substack{(\bar\Gamma, \bar\gamma)\;,\; \bar\delta \subseteq \bar\gamma }} ( \bar \Gamma, \bar\delta )  \otimes (\bar\Gamma / \bar\delta, \bar\gamma / \bar\delta)^{(1)} \otimes (\bar\Gamma / \bar\delta, \bar\gamma / \bar\delta)^{(2)}.
\end{eqnarray*}
\begin{eqnarray*}
(m^{13} \otimes I) \circ (\Delta \otimes \Delta)\circ \Phi (\bar \Gamma, \bar\gamma) &=& (m^{13} \otimes I)\left( \sum_{(\bar\Gamma, \bar\gamma)} \Delta\big((\bar\Gamma, \bar\gamma)^{(1)}\big) \otimes  \bar\Delta \big((\bar\Gamma, \bar\gamma)^{(2)}\big)\right) \\ 
&=& (m^{13} \otimes I) \big(\sum_{(\bar\Gamma , \bar\gamma), \bar\delta \subset \bar\gamma} (\bar\Gamma, \bar\delta)^{(1)}\otimes (\bar\Gamma/\bar\delta, \bar\gamma/\bar\delta)^{(1)} \otimes (\bar\Gamma, \bar\delta)^{(2)} \otimes (\bar\Gamma/\bar\delta, \bar\gamma/\bar\delta)^{(2)} \big)\\
&=& \sum_{(\bar\Gamma , \bar\gamma), \bar\delta \subset \bar\gamma} (\bar\Gamma, \bar\delta)^{(1)} (\bar\Gamma, \bar\delta)^{(2)}\otimes (\bar\Gamma/\bar\delta, \bar\gamma/\bar\delta)^{(1)} \otimes (\bar\Gamma/\bar\delta, \bar\gamma/\bar\delta)^{(2)}\\
&=& \sum_{(\bar\Gamma , \bar\gamma), \bar\delta \subset \bar\gamma} (\bar\Gamma, \bar\delta)\otimes (\bar\Gamma/\bar\delta, \bar\gamma/\bar\delta)^{(1)} \otimes (\bar\Gamma/\bar\delta, \bar\gamma/\bar\delta)^{(2)},
\end{eqnarray*}
which proves the theorem.
\end{proof}
\subsection{Relation between $(\wt {\Cal D'}_{\Cal T}, \bigstar, \Phi)$ and $(\wt {\Cal H'}_{\Cal T}, \star, \Psi)$ }
\begin{definition}
Let $\bar\Gamma_1 \in \wt {\Cal H'}_{\Cal T}$ and $(\bar\Gamma_2 , \bar\gamma_2) \in \wt {\Cal D'}_{\Cal T}$, we define two products $\trianglerightg$ and $\starg$ by:
\begin{eqnarray}
\bar\Gamma_1 \trianglerightg \bar\Gamma_2 &:=& \sum_{\substack{v \in \Cal V(\gamma_2)\; , v = \smop{res} \Gamma_1}}  \bar\Gamma_1 \trianglerightv \bar\Gamma_2,
\end{eqnarray}
\begin{eqnarray}
\bar\Gamma_1 \starg \bar\Gamma_2 &:=& \sum_{(\Gamma_1)}  (\bar\Gamma_1)^{(1)} (\bar\Gamma_1)^{(2)}\trianglerightg \bar\Gamma_2.
\end{eqnarray}
\end{definition}
\begin{remark} We remark that $\trianglerightg$ and $\starg$ are respectively the restrictions of $\triangleright$ and $\star$ on the set of vertices of subgraph $\gamma_2$, so they are respectively pre-Lie and associative.
\end{remark}
\begin{theorem} $(\wt {\Cal D'}_{\Cal T}, \bigstar, \Phi)$ is a module-bialgebra on $(\wt {\Cal H'}_{\Cal T}, \star, \Psi)$.
\end{theorem}

\begin{proof}
We consider the map: $\alpha: \wt {\Cal D'}_{\Cal T} \otimes \wt {\Cal H'}_{\Cal T} \longrightarrow \wt {\Cal D'}_{\Cal T}$ defined for all $(\Gamma , \gamma) \in \wt {\Cal D'}_{\Cal T}$ and $\Gamma' \in \wt {\Cal H'}_{\Cal T}$ by:
$$\alpha ((\bar\Gamma , \bar\gamma)  \otimes \bar\Gamma') = (\bar\Gamma \star \bar\Gamma', \bar\gamma).$$
To prove this theorem, first we will show that $\alpha$ is an action which results in the following commutative diagram:
\diagramme{
\xymatrix{
\wt{D'}_{\Cal T} \otimes \wt{\Cal H'}_{\Cal T} \otimes \wt{\Cal H'}_{\Cal T}  \ar[d]_{I\otimes \star} \ar[rrr]^{\alpha \otimes I} 
&&&\wt{D'}_{\Cal T} \otimes \wt{\Cal H'}_{\Cal T} \ar[d]^{\alpha}\\
\wt{D'}_{\Cal T} \otimes \wt{\Cal H'}_{\Cal T} \ar[rrr]_{\alpha} 
&&& \wt{\Cal D'}_{\Cal T} }
}
Let $(\Gamma , \gamma) \in \wt {\Cal D'}_{\Cal T}$ and $\bar\Gamma_1, \bar\Gamma_2 \in \wt {\Cal H'}_{\Cal T}$, we have:
\begin{eqnarray*}
\alpha \circ (\alpha \otimes I) [(\bar \Gamma, \bar \gamma)\otimes \bar\Gamma_1 \otimes \bar\Gamma_2]  &=& \alpha [(\bar \Gamma \star \bar\Gamma_1, \bar\gamma) \otimes \bar\Gamma_2]\\ 
&=& \big((\bar \Gamma \star \bar\Gamma_1) \star \bar\Gamma_2, \bar\gamma\big) \\
&=& \big(\bar \Gamma \star (\bar\Gamma_1 \star \bar\Gamma_2), \bar\gamma\big)\\
&=& \alpha [(\bar \Gamma , \bar\gamma) \otimes \bar\Gamma_1 \star \bar\Gamma_2]\\
&=& \alpha \circ (I \otimes \star) [(\bar \Gamma , \bar\gamma) \otimes \bar\Gamma_1 \otimes \bar\Gamma_2].
\end{eqnarray*}
Second, we show that the following diagram is commutative:
\diagramme{
\xymatrix{
 \wt {\Cal D'}_{\Cal T} \otimes  \wt {\Cal D'}_{\Cal T}  \otimes  \wt {\Cal H'}_{\Cal T} \ar[d]_{I\otimes I \otimes \Psi} \ar[rr]^{\bigstar \otimes I} 
&& \wt {\Cal D'}_{\Cal T} \otimes  \wt {\Cal H'}_{\Cal T} \ar[d]^{\alpha}\\
   \wt {\Cal D'}_{\Cal T} \otimes  \wt {\Cal D'}_{\Cal T}  \otimes  \wt {\Cal H'}_{\Cal T} \otimes  \wt {\Cal H'}_{\Cal T} \ar[d]_{\tau^{23}}&& \wt {\Cal D'}_{\Cal T} \\
\wt {\Cal D'}_{\Cal T} \otimes  \wt {\Cal H'}_{\Cal T} \otimes  \wt {\Cal D'}_{\Cal T} \otimes  \wt {\Cal H'}_{\Cal T}\ar[rr]_{\alpha \otimes \alpha} 
&&\ar[u]_{\bigstar} \wt {\Cal D'}_{\Cal T} \otimes  \wt {\Cal D'}_{\Cal T}
 }
}
We denoted by: $\Psi^{23} = \tau^{23} \circ (I\otimes I \otimes \Psi).$
\begin{eqnarray*}
\bigstar \circ (\alpha \otimes \alpha) \circ \Psi^{23}\big((\bar \Gamma_1, \bar \gamma_1)\otimes (\bar \Gamma_2, \bar \gamma_2)\otimes \bar\Gamma\big)  &=&\sum_{(\Gamma)} \alpha \big((\bar \Gamma_1 , \bar\gamma_1)\otimes \bar\Gamma^{(1)}\big) \bigstar \alpha \big((\bar \Gamma_2 , \bar\gamma_2)\otimes \bar\Gamma^{(2)}\big)\\ 
&=& \sum_{(\Gamma)} (\bar \Gamma_1 \star \bar\Gamma^{(1)}, \bar\gamma_1) \bigstar (\bar \Gamma_2 \star \bar\Gamma^{(2)}, \bar\gamma_2)\\ 
&=& \sum_{\substack{(\Gamma) \\v \in \Gamma_2-\gamma_2}} \big((\bar \Gamma_1 \star \bar\Gamma^{(1)})^{(1)} (\bar \Gamma_1 \star \bar\Gamma^{(1)})^{(2)} \trianglerightv(\bar \Gamma_2 \star \bar\Gamma^{(2)}), \bar\gamma_1\bar\gamma_2\big)\\ 
&=& \sum_{\substack{(\Gamma)}} \big((\bar \Gamma_1 \star \bar\Gamma^{(1)})\star (\bar \Gamma_2 \star \bar\Gamma^{(2)}) - (\bar \Gamma_1 \star \bar\Gamma^{(1)})\starg (\bar \Gamma_2 \star \bar\Gamma^{(2)}), \bar\gamma_1\bar\gamma_2\big)\\
&=& \big((\bar \Gamma_1 \star \bar\Gamma_2)\star\bar\Gamma - (\bar \Gamma_1 \starg \bar\Gamma_2)\star \bar\Gamma, \bar\gamma_1\bar\gamma_2\big).
\end{eqnarray*}
On the other hand:
\begin{eqnarray*}
\alpha \circ (\bigstar \otimes I) \big((\bar \Gamma_1, \bar \gamma_1)\otimes (\bar \Gamma_2, \bar \gamma_2)\otimes \bar\Gamma\big)  &=& \alpha \big((\bar \Gamma_1, \bar \gamma_1)\bigstar (\bar \Gamma_2, \bar \gamma_2)\otimes \bar\Gamma\big)\\ 
&=& \alpha \big((\bar \Gamma_1 \star \bar \Gamma_2 - \bar \Gamma_1 \star \bar \Gamma_2 , \bar \gamma_1 \bar \gamma_2)\otimes \bar\Gamma\big)\\ 
&=& \big((\bar \Gamma_1 \star \bar \Gamma_2 - \bar \Gamma_1 \starg \bar \Gamma_2) \star \Gamma, \bar \gamma_1 \bar \gamma_2\big)\\ 
&=& \big((\bar \Gamma_1 \star \bar \Gamma_2)\star \bar\Gamma - (\bar \Gamma_1 \starg \bar \Gamma_2) \star \bar\Gamma, \bar \gamma_1 \bar \gamma_2\big).
\end{eqnarray*}
Finally, we prove that the coproduct $\Phi$ is morphism of modules. This amounts to the commutativity of the following diagram:
\diagramme{
\xymatrix{
 \wt {\Cal D'}_{\Cal T} \otimes \wt {\Cal H'}_{\Cal T} \ar[d]_{I\otimes \Psi} \ar[rr]^{\alpha} 
&& \wt {\Cal D}_{\Cal T}\ar[d]^{\Phi}\\
  \wt {\Cal D'}_{\Cal T} \otimes  \wt {\Cal H'}_{\Cal T}  \otimes  \wt {\Cal H'}_{\Cal T} \ar[d]_{\Phi \otimes I \otimes I}&& \wt {\Cal D}_{\Cal T} \otimes  \wt {\Cal D'}_{\Cal T}\\
 \wt {\Cal D'}_{\Cal T} \otimes  \wt {\Cal D'}_{\Cal T} \otimes  \wt {\Cal H'}_{\Cal T} \otimes  \wt {\Cal H'}_{\Cal T} \ar[rr]_{\tau^{23}} 
&&\ar[u]_{\alpha \otimes \alpha}  \wt {\Cal D'}_{\Cal T} \otimes  \wt {\Cal H'}_{\Cal T} \otimes   \wt {\Cal D'}_{\Cal T} \otimes  \wt {\Cal H'}_{\Cal T}
 }
}
Let $(\Gamma_1 , \gamma_1) \in \wt {\Cal D'}_{\Cal T}$ and $\bar\Gamma_2 \in \wt {\Cal H'}_{\Cal T}$, we have:
\begin{eqnarray*}
\Phi \circ \alpha \big((\bar \Gamma_1, \bar \gamma_1)\otimes \bar\Gamma_2\big) &=& \Phi \big(\bar \Gamma_1 \star \bar\Gamma_2, \bar \gamma_1 \big)\\ 
&=& \sum_{\substack{(\Gamma_1)}}  \Phi  \big(\bar \Gamma^{(1)}_1 (\bar \Gamma^{(2)}_1 \triangleright \bar \Gamma_2) , \bar\gamma \big)  \\
&=& \sum_{\substack{(\bar\Gamma_1, \bar\gamma_1)\;,\; (\bar\Gamma_2)}} (\bar \Gamma^{(11)}_1 (\bar \Gamma^{(12)}_1 \triangleright \bar \Gamma^{(1)}_2) , \bar\gamma^{(1)})   \otimes (\bar \Gamma^{(21)}_1 (\bar \Gamma^{(22)}_1 \triangleright \bar \Gamma^{(2)}_2) , \bar\gamma^{(2)})\\
&=& \sum_{\substack{(\bar\Gamma_1, \bar\gamma_1)\;,\; (\bar\Gamma_2)}} (\bar \Gamma^{(1)}_1 \star  \bar \Gamma^{(1)}_2 , \bar\gamma^{(1)})   \otimes (\bar \Gamma^{(2)}_1 \star \bar \Gamma^{(2)}_2) , \bar\gamma^{(2)}).
\end{eqnarray*}
We use the notation: $(\Phi \otimes I \otimes I) \circ (I \otimes \Psi) = \Phi \otimes \Psi$. 
\begin{eqnarray*}
(\alpha\otimes \alpha) \circ \tau^{23}\circ (\Phi \otimes \Psi) \big((\bar \Gamma_1, \bar \gamma_1)\otimes \bar\Gamma_2\big)
&=& (\alpha\otimes \alpha) \big(\sum_{\substack{(\bar\Gamma_1, \bar\gamma_1)\;,\; (\bar\Gamma_2)}} (\bar\Gamma^{(1)}_1, \bar\gamma^{(1)}_1)\otimes  \bar\Gamma^{(1)}_2 \otimes (\bar\Gamma^{(2)}_1, \bar\gamma^{(2)}_1) \otimes  \bar\Gamma^{(2)}_2\big)\\ 
 &=& \sum_{\substack{(\bar\Gamma_1, \bar\gamma_1)\;,\; (\bar\Gamma_2)}} \alpha\big((\bar\Gamma^{(1)}_1, \bar\gamma^{(1)}_1)\otimes  \bar\Gamma^{(1)}_2 \big)\otimes \alpha\big((\bar\Gamma^{(2)}_1, \bar\gamma^{(2)}_1) \otimes  \bar\Gamma^{(2)}_2)\big)\\ 
&=& \sum_{\substack{(\bar\Gamma_1, \bar\gamma_1)\;,\; (\bar\Gamma_2)}} (\bar \Gamma^{(1)}_1 \star  \bar \Gamma^{(1)}_2 , \bar\gamma^{(1)})   \otimes (\bar \Gamma^{(2)}_1 \star \bar \Gamma^{(2)}_2) , \bar\gamma^{(2)}),
\end{eqnarray*}
which proves the theorem.
\end{proof}



\begin{thebibliography}{abcdsfgh}
{\small{
\bibitem{SA} S. Agarwala, \textit{The geometry of renormalisation}, PhD thesis, Johns Hopkins University (2009).
\bibitem{DMB} M. Belhaj Mohamed, D. Manchon, \textsl{Bialgebra of specified graphs and external structures}, Ann. Inst. Henri Poincar\'e, D, Volume 1, Issue 3, pp. 307-335 (2014).
\bibitem{DB} M. Belhaj Mohamed, D. Manchon, \textsl{Doubling bialgebras of rooted trees}, Lett Math Phys 107-145 (2017). 
\bibitem{mb} M. Belhaj Mohamed, \textsl{Doubling bialgebras of graphs and Feynman rules}, Confluentes Mathematici, Vol. 8 no. 1 (2016), p. 3-30.
\bibitem{db} D. Burde, \textsl{Left-symmetric algebras, or pre-Lie algebras in geometry
and physics}, Central European Journal of Mathematics 4, Nr. 3, 323-357 (2006).
\bibitem{ch} F. Chapoton, \textsl{Alg\`ebres pr\'e-Lie et alg\`ebres de Hopf li\'ees \`a la renormalisation}, Comptes-
Rendus Acad. Sci., 332 S\'erie I (2001), 681-684.
\bibitem{cl} F. Chapoton, M. Livernet, \textit{Pre-Lie algebras and the rooted trees operad}, Internat. Math. Res. Notices 8 (2001) 395-408.
\bibitem{A.D2000}A. Connes, D. Kreimer, \textit{Renormalization in quantum field theory and the Riemann-Hilbert problem. I. The Hopf algebra structure of graphs and the main theorem}, Comm. Math. Phys. 210, $n^{\circ} 1$, 249-273 (2000).
\bibitem{ad01}A. Connes, D. Kreimer, \textit{Renormalization in quantum field theory and the Riemann-Hilbert problem. II. The $\beta$-function,
diffeomorphisms and the renormalization group}, Comm. in Math. Phys. 216, 215-241 (2001).
\bibitem{ad98}A. Connes, D. Kreimer, \textit{Hopf algebras, renormalization and noncommutative geometry}, Comm. in Math. Phys. 199,203-242 (1998).
\bibitem{CM}A. Connes, M. Marcolli, \textsl{Noncommutative geometry, quantum fields and motives}, preprint, http://www.alainconnes.org/fr/bibliography.php.
\bibitem{dk98}D. Kreimer, \textit{On the Hopf algebra structure of perturbative quantum field theories}, Adv. Theor. Math. Phys. 2, 303-334(1998).
\bibitem{kp} K. Ebrahimi-Fard, F. Patras, \textsl{ The pre-Lie algebras structure of the time-ordred exponential}, Lett Math Phys (2014) 104: 1281-1302.
\bibitem{loro} J.~L. Loday and M.O. Ronco, \textit{Combinatorial Hopf algebras}, Quanta of maths, Clay Math. Proc. 11, 347-383, Amer. Math. Soc., Providence, RI, 2010.
\bibitem{dm11}D. Manchon, \textit{A short survey on pre-Lie algebras}, E. Schrodinger Institut Lectures in Math. Phys., Eur. Math. Soc, A.Carey Ed. (2011).
\bibitem{Dm11}D. Manchon, \textit{On bialgebra and Hopf algebra of oriented graphs}, Confluentes Math. Volume 04, No. 1 (2012). 
\bibitem{Dm08}D. Manchon, \textit{Hopf algebras and renormalisation}, Handbook of Algebra,Vol. 5 (M. Hazewinkel ed.), 365-427 (2008).
\bibitem{og} J.~M. Oudom and D. Guin, \textit{On the Lie envelopping algebra of a pre-Lie algebra}, Journal of K-theory: K-theory and its Applications to Algebra, Geometry, and Topology, pp. 147-167, (2008).
\bibitem{Sw69}M.~E.~Sweedler, \textit{Hopf algebras}, Benjamin, New-York (1969).
\bibitem{EP} E. Panzer, \textit{Hopf algebraic renormalization of Kreimer's toy model}. Arxiv: math.QA: 1202.3552v1 (2012).
\bibitem{wvs}W.D. van Suijlekom, \textit{The structure of renormalization Hopf algebras for gauge theories I: representing Feynman
graphs on BV-algebras}, Commun. Math. Phys. 290  291-319 (2009).
\bibitem{wvs06}W.D. van Suijlekom, \textit{The Hopf algebra of Feynman graphs in QED}, letters in Math. Phys. 77, 265-281 (2006). 
}
}
\end{thebibliography}
\end{document}